\documentclass{amsart}
\usepackage[all]{xy}
\usepackage{amsmath}
\usepackage{amssymb}
\usepackage{amsthm}
\usepackage{amscd}
\newcounter{pcounter}
\newcommand{\del}{\partial}
\newcommand{\R}{\mathcal{R}}
\newcommand{\E}{\mathcal{E}}

\newcommand{\ZZ}{\Bbb Z}
\newcommand{\RR}{\Bbb R}

\newcommand{\NN}{\Bbb N}
\newcommand{\QQ}{\Bbb Q}
\newcommand{\CC}{\Bbb C}
\newcommand{\ip}[1]{\langle #1 \rangle}

\newcommand{\widetidle}{\widetilde}
\newcommand{\varespilon}{\varepsilon}

\newcommand{\actson}{\curvearrowright}

\newtheorem{conj}{Conjecture}
\newtheorem{theorem}{Theorem}
\newtheorem{definition}[theorem]{Definition}
\newtheorem{proposition}[theorem]{Proposition}
\newtheorem{cor}[theorem]{Corollary}
\newtheorem{lemma}[theorem]{Lemma}

\newcommand{\FF}{\Bbb F}

\DeclareMathOperator{\vdim}{vdim}
\DeclareMathOperator{\Vect}{Vect}

\newcommand{\mult}{\textnormal{mult}}

\DeclareMathOperator{\sgn}{sgn}
\DeclareMathOperator{\Hamm}{Hamm}
\DeclareMathOperator{\vol}{vol}

\DeclareMathOperator{\Isom}{Isom}
\DeclareMathOperator{\Sym}{Sym}

\DeclareMathOperator{\Span}{Span}

\DeclareMathOperator{\id}{Id}

\DeclareMathOperator{\Ball}{Ball}

\DeclareMathOperator{\tr}{tr}
\DeclareMathOperator{\dom}{dom}
\DeclareMathOperator{\ran}{ran}

\DeclareMathOperator{\Proj}{Proj}

\DeclareMathOperator{\Hom}{Hom}

\DeclareMathOperator{\Tr}{Tr}

\DeclareMathOperator{\wk}{wk}
\numberwithin{theorem}{section}

\begin{document}
\title[$l^{p}$ Dimension for Banach Space Representations of Sofic Groups II] {An $l^{p}$-Version of Von Neumann Dimension for Banach Space Representations of Sofic Groups II }        
\author{Ben Hayes}
\address{UCLA Math Sciences Building\\
         Los Angeles,CA 90095-1555}
\email{brh6@ucla.edu}
\date{\today}
\maketitle
\tableofcontents

\section{Introduction}

	This paper is intended as a follow up paper to \cite{Me}. Let us first recall some necessary notation and definitions.
\begin{definition} \emph{Let $V$ be a Banach space. We shall use $\Isom(V)$ for the group of all linear, surjective, isometric maps from $V$ to itself}.\end{definition}

\begin{definition} \emph{ For $\sigma,\tau\in S_{n}$ (here $S_{n}$ is the group of self-bijections of $\{1,\dots,n\}$) we define the} Hamming distance \emph{by}\end{definition}
\[d_{\Hamm}(\sigma,\tau)=\frac{1}{n}|\{j:\sigma(j)\ne \tau(j)\}|.\]
The Hamming distance can also be seen as the probability that $\sigma\ne \tau,$ using the uniform probability measure on $\{1,\dots,n\}.$

\begin{definition} \emph{Let $\Gamma$ be a countable discrete group. A} sofic approximation \emph{ of $\Gamma$ is a sequence $\sigma_{i}\colon \Gamma\to S_{d_{i}}$  of functions, not assumed to be homomorphisms, with $\sigma_{i}(e)=1$ such that}
\[d_{\Hamm}(\sigma_{i}(gh),\sigma_{i}(g)\sigma_{i}(h))\to 0,\mbox{\emph{ for all $g,h\in \Gamma$,}}\]
\[d_{\Hamm}(\sigma_{i}(g),\sigma_{i}(h))\to 1,\mbox{\emph{ for all $g\ne h$ in $\Gamma.$}}\]
\emph{We say that $\Gamma$ is \emph{sofic}, if it has a sofic approximation.}
\end{definition}
\begin{definition} \emph{On $M_{n}(\CC)$ we shall use $\tr=\frac{1}{n}\Tr,$ where $\Tr$ is the usual trace. We shall use $\ip{A,B}=\tr(B^{*}A)$ and $\|A\|_{p}=\tr((A^{*}A)^{p/2})^{1/p}.$ We shall use $\|A\|_{\infty}$ for the operator norm of $A.$ }\end{definition}

	In fact, if $M$ is a von Neumann algebra with a faithful normal tracial state $\tau,$ we will use
\[\|x\|_{p}=\tau((x^{*}x)^{p/2})^{1/p}\]
and $\|x\|_{\infty}$ for the operator norm of $x\in M.$

\begin{definition}\emph{ Let $\Gamma$ be a countable discrete group. An} embedding sequence \emph{of $\Gamma$ is a sequence $\sigma_{i}\colon \Gamma\to U(d_{i}),$ (here $U(n)$ is the unitary group of $\CC^{n}$) satsifying $\sigma_{i}(e)=\id$ and}
\[\|\sigma_{i}(gh)-\sigma_{i}(g)\sigma_{i}(h)\|_{2}\to 0\mbox{\emph{ for all $g,h\in \Gamma$}}\]
\[\ip{\sigma_{i}(g),\sigma_{i}(h)}\to 0\mbox{\emph{ for all $g\ne h\in\Gamma$}}.\]\end{definition}
We say that $\Gamma$ is\emph{ $\R^{\omega}$-embeddable } if it has an embedding sequence. The terminology comes from the Connes' Embedding Problem, as one can show that $\Gamma$ is $\R^{\omega}$-embeddable according to our definition if and only if the group von Neumann algebra of $\Gamma$ embeds into an ultrapower of the hyperfinite II$_{1}$ factor $\R$ (see \cite{Radul} Proposition 2.5 and Definition 2.6). It is a simple exercise to show that every sofic group is $\R^{\omega}$-embeddable. The class of sofic groups include amenable groups, residually finite groups, and is closed under free products with amalgamation over amenable groups, increasing unions, and taking subgroups (see \cite{ESZ1} for proofs of these facts). In particular, all linear groups are sofic.

 In \cite{Me} if we are given a countable discrete group $\Gamma$ and $\Sigma=(\sigma_{i}\colon \Gamma\to \Isom(V_{i}))$ with $\dim(V_{i})<\infty,$ then to every uniformly bounded action of $\Gamma$ on a Banach space $V,$ we defined a number
\[\dim_{\Sigma}(V,\Gamma)\in [0,\infty].\]
The definition will be given later in this section. Suppose $\Gamma$ is a sofic group, and  $V_{i}=l^{p}(d_{i}), $ and $\Sigma=(\sigma_{i}\colon\Gamma\to S_{d_{i}})$ is as sofic approximation. Let $\widetilde{\Sigma}=(\widetilde{\sigma}_{i}\colon\Gamma\to \Isom(l^{p}(d_{i})))$ be the sequence of maps defined by composing $\sigma_{i}$ with the map $S_{d_{i}}\to \Isom(l^{p}(d_{i}))$ given by the natural action of $S_{d_{i}}$ on $l^{p}(d_{i}).$  In this case we will use
\[\dim_{\Sigma,l^{p}}(V,\Gamma)\]
instead of
\[\dim_{\widetidle{\Sigma}}(V,\Gamma).\]
Similarly if $\Gamma$ is $\R^{\omega}$-embeddable and $\Sigma=(\sigma_{i}\colon\Gamma\to U(d_{i}))$ is an embedding sequence, we let
\[\dim_{\Sigma,S^{p},\textnormal{mult}}(V,\Gamma),\]
\[\dim_{\Sigma,S^{p},\textnormal{conj}}(V,\Gamma),\]
be the dimensions coming from the left multiplication and conjugation actions of $U(n)$ on $S^{p}(n),$ respectively.

	In \cite{Me}, we proved the following properties of this dimension function
\begin{list}{Property \arabic{pcounter}:~}{\usecounter{pcounter}}
\item $\dim_{\Sigma}(W,\Gamma)\leq \dim_{\Sigma}(V,\Gamma)$ if there is a equivariant bounded linear map $W\to V$ with dense image.
\item $\dim_{\Sigma}(W,\Gamma)\leq \dim_{\Sigma}(V,\Gamma)+\dim_{\Sigma}(V/W,\Gamma),$ if $W\subseteq V$ is a closed $\Gamma$-invariant subspace.
\item $\dim_{\Sigma,l^{p}}(W\oplus V,\Gamma)\geq \dim_{\Sigma,l^{p}}(W,\Gamma)+\underline{\dim}_{\Sigma,l^{p}}(V,\Gamma)$ for $2\leq p<\infty,$ where $\underline{\dim}$ is a ``lower dimension," and is also an invariant,
\item $\dim_{\Sigma,l^{p}}(l^{p}(\Gamma,V),\Gamma)=\underline{\dim}_{\Sigma,l^{p}}(l^{p}(\Gamma,V))=\dim(V)$ for $1\leq p\leq 2,$
\item $\underline{\dim}_{\Sigma,l^{p}}(W,\Gamma)\geq \dim_{L(\Gamma)}(\overline{W}^{\|\cdot\|_{2}})$ if $1\leq p\leq 2,$ $W\subseteq l^{p}(\NN,l^{p}(\Gamma)),$
\item  $\underline{\dim}_{\Sigma,S^{p},conj}(W,\Gamma)\geq \dim_{L(\Gamma)}(\overline{W}^{\|\cdot\|_{2}})$ if $1\leq p\leq 2,$ $W\subseteq l^{p}(\NN,l^{p}(\Gamma)),$
\item $\underline{\dim}_{\Sigma,l^{2}}(H,\Gamma)=\dim_{\Sigma,l^{2}}(H,\Gamma)=\dim_{L(\Gamma)}H$ if $H\subseteq l^{2}(\NN,l^{2}(\Gamma)).$

\end{list}

	Thus $\dim_{\Sigma,l^{p}}$ can be seen as an extension of the von Neumann dimension of a $\Gamma$-invariant subspace of $l^{2}(\NN,\Gamma)$ as defined by Murray and von Neumann. The above shows that $\dim_{\Sigma,l^{p}}$ has many of the properties that the usual dimension in linear algebra and von Neumann dimension have and thus it makes sense to think of $\dim_{\Sigma,l^{p}}$ as a version of von Neumann dimension. We mention that in \cite{MP} Monod-Petersen show that if $2<p<\infty,$ then any isomorphism invariant associated to $\Gamma$-invariant subspaces of $l^{p}(\Gamma)$ cannot satisfy all the  properties that von Neumann dimension satisfies. In particular, they show that if $\Gamma$ contains an infinite elementary amenable subgroup and $2<p<\infty,$ then there exists closed $\Gamma$-invariant linear subspaces $E_{n}$ and $F\ne \{0\}$ of $l^{p}(\Gamma)$ with $E_{n}\cap F=\{0\}$ for all $n,$ but
\[l^{p}(\Gamma)=\overline{\bigcup_{n=1}^{\infty}E_{n}}.\]
This is impossible for $p=2,$ because
\[\dim_{L(\Gamma)}(\mathcal{H}_{n})\to \dim_{L(\Gamma)}\left(\overline{\bigcup_{n=1}^{\infty}\mathcal{H}_{n}}\right),\]
whenever $\mathcal{H}_{n}$ is an increasing sequence of closed linear $\Gamma$-invariant subspaces of $l^{2}(\Gamma)$.

	Let us mention how the results of Monod-Petersen show that there must be paradoxical properties of our dimension function. Let $V_{n},W$ be the polars of $E_{n},F$ constructed by Monod-Petersen. We see that if  $1<p<2,$ then we have a $\Gamma$-invariant closed subspace $W\subseteq l^{p}(\Gamma)$ with $W\ne l^{p}(\Gamma)$ and  decreasing closed $\Gamma$-invariant subspaces $V_{n}\subseteq l^{p}(\Gamma)$ such that $\overline{W+V_{n}}^{\|\cdot\|_{p}}=l^{p}(\Gamma)$ and $\bigcap_{n=1}^{\infty}V_{n}=\{0\}.$ We thus have a $\Gamma$-equivariant map with dense image
\[V_{n}\oplus W\to l^{p}(\Gamma),\]
so
\[1\leq \dim_{\Sigma,l^{p}}(V_{n},\Gamma)+\underline{\dim}_{\Sigma,l^{p}}(W,\Gamma).\]
Thus one of two things occurs. Either
\[\liminf_{n\to \infty}\dim_{\Sigma,l^{p}}(V_{n},\Gamma)>0\]
or
\[\underline{\dim}_{\Sigma,l^{p}}(W,\Gamma)\geq 1.\]
Therefore, one of the two properties
\[\dim_{\Sigma,l^{p}}(W,\Gamma)< \dim_{\Sigma,l^{p}}(V,\Gamma)\mbox{ if $W$ is a proper closed linear subspace of a Banach space $V$},\]
\[\dim_{\Sigma,l^{p}}(V_{n},\Gamma)\to 0,\textnormal{ if $V_{n}\supseteq V_{n+1}$ are  Banach spaces with $\bigcap_{n=1}^{\infty}V_{n}=\{0\}$},\]
must fail for $l^{p}$-dimension. Each of these properties is true for von Neumann dimension, and so we must have some paradoxical properties (i.e. contrary to what one would intuitively believe should hold for a dimension function) for $l^{p}$ dimension when $p\ne 2.$

\begin{definition}\label{D:Hom} \emph{Let $\Gamma$ be a countable discrete group and $\Sigma=(\sigma_{i}\colon \Gamma\to \Isom(V_{i})).$ Let $\Gamma$ have a uniformly bounded action on a Banach space $V$ and let $S=(x_{j})_{j=1}^{\infty}$ be a bounded sequence in $V.$ For $F\subseteq \Gamma$ finite and $m\in \NN,$ let 
\[V_{F,m}=\Span\{g_{1}g_{2}\dots g_{k}x_{j}:g_{1},\dots,g_{k}\in F,1\leq j,k\leq m\}.\]
For $M,\delta>0,$ we let $\Hom_{\Gamma}(S,F,m,\delta,\sigma_{i})_{M}$ consist of all bounded linear maps $T\colon V_{F,m}\to V_{i}$ such that $\|T\|\leq M$ and}
\[\|T(g_{1}\cdots g_{k}x_{j})-\sigma_{i}(g_{1})\cdots\sigma_{i}(g_{k})T(x_{j})\|<\delta\mbox{ \emph{for all $g_{1},\dots,g_{k}\in F,1\leq j,k\leq m$}}.\]
\emph{We shall typically denote $\Hom_{\Gamma}(S,F,m,\delta,\sigma_{i})_{1}$ by $\Hom_{\Gamma}(S,F,m,\delta,\sigma_{i}).$ If $\Sigma$ is a sofic approximation, we use $\Hom_{\Gamma,p}(S,F,m,\delta,\sigma_{i})$ for the space of maps above using the permutation action of $S_{d_{i}}$ on $l^{p}(d_{i}).$}\end{definition}

\begin{definition}\emph{ Let $V$ be a vector space with pseudonorm $\rho.$ For $A\subseteq V$ and $\varepsilon>0$ we say that a linear subspace $W\subseteq V$} $\varepsilon$-contains \emph{$A,$  written $A\subseteq_{\varepsilon}W,$ if for all $x\in A,$ there is a $w\in W$ such that $\rho(x-w)<\varepsilon.$ We set $d_{\varepsilon}(A,\rho)$ to be the smallest dimension of a linear subspace which $\varepsilon$-contains $A.$ }\end{definition}
\begin{definition}\emph{A} product norm\emph{ on $l^{\infty}(\NN)$ is a norm $\rho$ with $\rho(f)\leq \rho(g)$ if $|f|\leq |g|,$ and such that $\rho$ induces the topology of pointwise convergence on $\{f\in l^{\infty}(\NN):\|f\|_{\infty}\leq 1\}.$ If $\rho$ is a product norm, and $V$ is a Banach space we define $\rho_{V}$ on $l^{\infty}(\NN,V)$ by $\rho_{V}(f)=\rho((\|f(n)\|)_{n=1}^{\infty}).$} \end{definition}

\begin{definition}\emph{Let $\Gamma,S=(x_{j})_{j=1}^{\infty},\Sigma,V$ be as Definition \ref{D:Hom}. For $F\subseteq \Gamma$ finite and $m\in \NN$ let $\alpha_{S}\colon B(V_{F,m},V_{i})\to l^{\infty}(\NN,V_{i})$ be given by $\alpha_{S}(T)(n)=T(x_{j}),$ if $1\leq n\leq m$ and $\alpha_{S}(T)(n)=0$ for $n>m.$  We define}
\[\dim_{\Sigma}(S,F,m,\delta,\varepsilon,\rho)=\limsup_{i\to \infty}\frac{1}{\dim(V_{i})}d_{\varepsilon}(\alpha_{S}(\Hom_{\Gamma}(S,F,m,\delta,\sigma_{i})),\rho_{V_{i}}).\]
\[\dim_{\Sigma}(S,\varepsilon,\rho)=\limsup_{(F,m,\delta)}d_{\varepsilon}(\alpha_{S}(\Hom_{\Gamma}(S,F,m,\delta,\sigma_{i})),\rho_{V_{i}}).\]
\[\dim_{\Sigma}(S,\rho)=\sup_{\varepsilon>0}\dim_{\Sigma}(S,\varepsilon,\rho).\]
\end{definition}
	Here the triples $(F,m,\delta)$ are ordered by $(F,m,\delta)\leq (F',m',\delta')$ if $F\subseteq F',m\leq m',\delta'<\delta.$ In \cite{Me} it is shown that
\[\dim_{\Sigma}(S,\rho)=\dim_{\Sigma}(S',\rho')\]
if $\overline{\Span(\Gamma S)}=\overline{\Span(\Gamma S')},$ and $\rho,\rho'$ are product norms. Because of this we use $\dim_{\Sigma}(V,\Gamma)$ for $\dim_{\Sigma}(S,\rho)$ if $\overline{\Span(\Gamma S)}=V$ and $\rho$ is a product norm.  Also in \cite{Me} we showed that
\[\dim_{\Sigma}(V,\Gamma)=\sup_{\varepsilon>0}\liminf_{(F,m,\delta)}\limsup_{i\to \infty}\frac{1}{\dim(V_{i})}d_{\varepsilon}(\alpha_{S}(\Hom_{\Gamma}(S,F,m,\delta,\sigma_{i})),\rho_{V_{i}}).\]
It is a simple exercise to show that $\Hom_{\Gamma}(S,F,m,\delta,\sigma_{i})$ may be replaced by $\Hom_{\Gamma}(S,F,m,\delta,\sigma_{i})_{M}$ for any $M>0.$
We let $\underline{\dim}_{\Sigma}(V,\Gamma)$ be the number obtained by replacing the first limit supremum with a limit infimum (again we showed in \cite{Me} that this depends only on $\overline{\Span(\Gamma S)}$). Similar to above we showed in \cite{Me} that
\[\underline{\dim}_{\Sigma}(V,\Gamma)=\sup_{\varepsilon>0}\liminf_{(F,m,\delta)}\liminf_{i\to \infty}\frac{1}{\dim(V_{i})}d_{\varepsilon}(\alpha_{S}(\Hom_{\Gamma}(S,F,m,\delta,\sigma_{i}),\rho_{V_{i}}).\]

	These definitions may seem quite technical and bizarre, but they are really inspired by ideas of Bowen \cite{Bow}, Kerr and Li \cite{KLi}, Gournay \cite{Gor} and Voiculescu \cite{Voi}. The main point is that one should view the usual von Neumann dimension as a type of dynamical entropy thinking of the action of $\Gamma$ on $l^{2}(\Gamma)$ as an analogue of a Bernoulli shift action of $\Gamma.$ Bowen in \cite{Bow} and Kerr and Li in \cite{KLi} give a microstates version of dynamical entropy for sofic groups. Similar to Kerr and Li, we consider ``almost structure-preserving maps" (in this case almost equivariant maps), and measure the growth rate of the size of the space of such maps. Here it makes sense to consider the linear growth rate, since $\varepsilon$-dimension can grow at most linearly. In this paper we prove some of the conjectures stated in \cite{Me}. Namely, we show the following new properties of $l^{p}$-dimension:

\begin{list}{Property \arabic{pcounter}:~}{\usecounter{pcounter}}
\item $\underline{\dim}_{\Sigma,l^{p}}(H_{1}^{l^{p}}(\FF_{n}),\FF_{n})=\dim_{\Sigma,l^{p}}(H_{1}^{l^{p}}(\FF_{n}),\FF_{n})=n-1.$
\item $\dim_{\Sigma,l^{p}}(H_{l^{p}}^{1}(\FF_{n}),\FF_{n})=\dim_{\Sigma,l^{p}}(H_{l^{p}}^{1}(\FF_{n}),\FF_{n})=n-1.$

\item $\dim_{\Sigma,S^{p},\textnormal{mult}}(\left(\bigoplus_{j=1}^{n}L^{p}(L(\Gamma)q_{j},\tau)\right),\Gamma)=\sum_{j=1}^{n}\tau(q_{j}),$ for $1\leq p<\infty,$ where $q_{1},\dots,q_{n}$ are projections in $L(\Gamma)$ and $\tau$ is the group trace.
\item $\dim_{\Sigma,l^{p}}(V,\Gamma)=0$ if $V$ is finite-dimensional and $\Gamma$ is infinite.

\end{list}

	Here $H^{1}_{l^{p}},H_{1}^{l^{p}}$ are $l^{p}$-homology and $l^{p}$-cohomology spaces, and $\FF_{n}$ is the free group on $n$ letters. Our approach to proving the last property is to consider a free, ergodic, probability measure-preserving action of $\Gamma$ such that the associated equivalence relation $\R_{\Gamma}$ is sofic ( the definition of what it means for an equivalence relation to be sofic will be given in Section \ref{S:trivial} and the Bernoulli action of a sofic group will be an example). Because $\R_{\Gamma}$ contains an amenable equivalence relation we can try to adapt the proof of the last property in the case $\Gamma=\ZZ.$ This suggests  exploring $l^{p}$-dimension for representations of equivalence relations, which is part of ongoing research (see \cite{Me2}). We will also give an equivalent approach to $l^{p}$-dimension defined by using vectors instead of almost equivariant operators.

\section{Triviality In The Case of Finite-Dimensional Representations}\label{S:trivial}

The main goal of this section is to prove the following.

\begin{theorem}\label{T:TrivialFinite} Let $\Gamma$ be a infinite sofic group, and $\Sigma$ a sofic approximation of $\Gamma.$  Then for every $1\leq p\leq \infty,$ and for uniformly bounded representation of $\Gamma$ on a finite-dimensional Banach space $V,$
\[\dim_{\Sigma,l^{p}}(V,\Gamma)=0.\]
\end{theorem}

Here is the outline of the proof. We will begin by studying $l^{p}$-dimension for amenable groups, using the standard technique of averaging over F\o lner sequences. Using this averaging technique, we show that for finite $\Gamma$
\[\dim_{\Sigma,l^{p}}(V,\Gamma)=\frac{\dim_{\CC}V}{|\Gamma|}.\]
This easily proves the theorem when $\Gamma$ has finite subgroups of unbounded size. We then show that
\[\dim_{\Sigma,l^{p}}(V,\ZZ)=0,\]
if $V$ is finite-dimensional. Since dimension increases when we restrict  the action to a subgroup, we may assume that $\Gamma$ has no elements of infinite order, but that there is a uniform bound on the size of a finite subgroup of $\Gamma.$ A compactness argument will show that $\Gamma$ has an infinite subgroup which acts on $V$ trivially, so we only have to show that
\[\dim_{\Sigma,l^{p}}(\CC,\Gamma)=0\]
where $\Gamma$ acts trivially on $\CC.$ To prove this last statement, we will pass to a sofic equivalence relation induced by the group and use that the full group of such an equivalence relation contains $\ZZ/n\ZZ$ for every integer $n.$

	We first show that in the case of an action of an amenable group, we may assume that the maps we use to compute dimension are only approximately equivariant after cutting down by certain subsets. We formalize this as follows.

\begin{definition} \emph{Let $\Gamma$ be a sofic group with a uniformly bounded action on a Banach space $V.$ Let $\sigma_{i}\colon \Gamma\to S_{d_{i}}$ be a sofic approximation. Fix  a bounded sequence $S=(a_{j})_{j=1}^{\infty}$ in $V.$ Let $A_{i}\subseteq \{1,\dots,d_{i}\}.$ For $F\subseteq \Gamma$ finite, $m\in \NN,\delta>0,$ we let $\Hom_{\Gamma,l^{p},(A_{i})}(S,F,m,\delta,\sigma_{i})$  be the set of all linear maps $T\colon V_{F,m}\to l^{p}(d_{i})$ such that $\|T\|\leq 1,$ and for all $1\leq j,k\leq m$ and all $s_{1},\dots,s_{k}\in F$ we have}
\[\|T(s_{1}\cdots s_{k}a_{j})-\sigma_{i}(s_{1})\cdots \sigma_{i}(s_{k})T(a_{j})\|_{l^{p}(A_{i})}<\delta.\]
\emph{Set}
\[\dim_{\Sigma,l^{p}}(S,\Gamma,(A_{i}),\varepsilon,F,\delta,\rho)=\limsup_{i\to \infty}\frac{1}{d_{i}}d_{\varepsilon}(\alpha_{S}(\Hom_{\Gamma,l^{p},(A_{i})}(S,F,m,\delta,\sigma_{i})),\rho_{l^{p}(d_{i})}),\]
\[\dim_{\Sigma,l^{p}}(S,\Gamma,(A_{i}),\varepsilon,\rho)=\inf_{\substack{ F\subseteq \Gamma \textnormal{finite},\\ \delta>0}}\dim_{\Sigma,l^{p}}(S,\Gamma,(A_{i}),\varepsilon,F,\delta,\rho),\]
\[\dim_{\Sigma,l^{p}}(S,\Gamma,(A_{i}),\rho)=\sup_{\varepsilon>0}\dim_{\Sigma,l^{p}}(S,\Gamma,(A_{i}),\varepsilon,\rho),\]
\emph{where $\rho$ is any product norm.}
\end{definition}

\begin{proposition} Fix a product norm $\rho$ on $l^{\infty}(\NN).$ Let $\Gamma$ be a countable amenable group and $\Sigma=(\sigma_{i}\colon \Gamma\to S_{d_{i}})$ a sofic approximation. Let $A_{i}\subseteq \{1,\dots,d_{i}\}$ be such that
\[\frac{|A_{i}|}{d_{i}}\to 1.\]
Then for any uniformly bounded action of $\Gamma$ on a separable Banach space $V,$ for every generating sequence $S$ in $V,$ for every product norm $\rho$ and every $1\leq p<\infty$ we have
\[\dim_{\Sigma,l^{p}}(V,\Gamma)=\dim_{\Sigma,l^{p}}(S,\Gamma,(A_{i}),\rho).\]
\end{proposition}

\begin{proof}Fix $S=(x_{j})_{j=1}^{\infty}$ a dynamically generating sequence for $V.$ As
\[\Hom_{\Gamma,l^{p}}(S,F,m,\delta,\sigma_{i})\subseteq \Hom_{\Gamma,l^{p},(A_{i})}(S,F,m,\delta,\sigma_{i})\]
for $m,i\in \NN,\delta>0$ and $F\subseteq \Gamma$ finite, we have
\[\dim_{\Sigma,l^{p}}(V,\Gamma)\leq \dim_{\Sigma,l^{p}}(S,\Gamma,(A_{i}),\rho).\]

	For the reverse inequality we first fix some notation. For $E, F$ finite subsets of $\Gamma$ containing the identity and $m\in \NN$ define
\[P_{i}^{(E)}\colon B(V_{E^{-1}F,m},l^{p}(d_{i}))\to B(V_{F,m},l^{p}(d_{i}))\]
by
\[P_{i}^{(E)}(T)=\frac{1}{|E|}\sum_{s\in E}\sigma_{i}(s)\circ T\circ s^{-1}.\]
Then $\|P_{i}^{(E)}\|\leq 1.$ Note that for $s_{1},\dots,s_{k}\in F$ and $T\in B(V_{F,k},l^{p}(d_{i}))$
\[P_{i}^{(E)}(T)(s_{1}\cdots s_{k}x)=\frac{1}{|E|}\sum_{s\in E}\sigma_{i}(s)T(s^{-1}s_{1}\cdots s_{k}x)=\]
\[\frac{1}{|E|}\sum_{s\in s_{k}^{-1}\cdots s_{1}^{-1}E}\sigma_{i}(s_{1}\cdots s_{k}s)T(s^{-1}x).\]
Let $B_{i}\subseteq \{1,\dots,d_{i}\}$ be the set of all $1\leq j\leq d_{i}$ such that
\[\sigma_{i}(s_{1}\cdots s_{k}s)^{-1}(j)=\sigma_{i}(s)^{-1}\sigma_{i}(s_{1}\cdots s_{k})^{-1}(j),\]
for all $s\in E,s_{1},\dots,s_{k}\in F,1\leq k\leq m.$ Then the above shows that if $T\in B(V_{E^{-1}F,m},l^{p}(B_{i}))$ then
\begin{equation}\label{E:MTMFE}
\|\sigma_{i}(s_{1}\cdots s_{k})P_{i}^{(E)}(T)(x_{l})-P_{i}^{(E)}(T)(s_{1}\cdots s_{k}x_{l})\|\leq 2\frac{|E\Delta s_{k}^{-1}\cdots s_{1}^{-1}E|}{|E|}\|T\|\|x_{l}\|,
\end{equation}
for $1\leq l\leq m.$

	Let $\varepsilon>0,$ and $M=\sup_{j}\|x_{j}\|<\infty.$  Since $\rho$ is a product norm, we may choose $N\in \NN,$ and $\kappa>0$ such that if $f,g\in l^{\infty}(\NN,l^{p}(d_{i}))$ and $\|f\|,\|g\|\leq M$ and
\[\max_{1\leq l\leq N}\|f(l)-g(l)\|_{p}<\kappa,\]
then
\[\rho(f-g)<\varepsilon.\]
Let $\delta>0$ depend upon $\kappa$ to be determined later. Let $m\geq \max(2,N)$ be an integer and let  $F$ be a symmetric finite subset of $\Gamma$ with $e\in F.$  Let $E\subseteq \Gamma$ be a finite set containing the identity, the set $E$ will depend upon $F,m,\delta$ in a manner to be determined later. Let $T\in \Hom_{\Gamma,l^{p},(A_{i})}(S,E^{-1}F,m,\delta,\sigma_{i}),$  then
\[P_{i}^{(E)}(\chi_{B_{i}}T)=\frac{1}{|E|}\sum_{s\in E}\sigma_{i}(s)\chi_{B_{i}}T\circ s^{-1}=\frac{1}{|E|}\sum_{s \in E}\chi_{\sigma_{i}(s)B_{i}}\sigma_{i}(s)T\circ s^{-1}.\]
Set $C_{i}=A_{i}\cap B_{i}\cap \bigcap_{s\in E}\sigma_{i}(s)(A_{i}\cap B_{i}).$ Then $\frac{|C_{i}|}{d_{i}}\to 1$ and for $1\leq j\leq m$
\begin{equation}\label{E:averageisFclose}
\|P_{i}^{(E)}(\chi_{B_{i}}T)(x_{j})-T(x_{j})\|_{l^{p}(C_{i})}\leq \frac{1}{|E|}\sum_{s\in E}\|\sigma_{i}(s)T(s^{-1}x_{j})-T(x_{j})\|_{l^{p}(\sigma_{i}(s)A_{i})}<\delta.
\end{equation}
By amenability of $\Gamma,$ we may choose $E$ with
\[\max_{\substack{1\leq k\leq m,\\ s_{1},\dots,s_{k}\in F}} 2\frac{|E\Delta s_{k}^{-1}\cdots s_{1}^{-1}E|}{|E|}\|x_{j}\|<\delta.\]
Then by (\ref{E:MTMFE}), we know $P_{i}^{(E)}(\chi_{B_{i}}(T))\in \Hom_{\Gamma,l^{p}}(S,F,m,\delta,\sigma_{i}).$ By (\ref{E:averageisFclose}),
\begin{equation}\label{E:YPMF}
\max_{1\leq j\leq m}\|\chi_{C_{i}}(P_{i}^{(E)}(\chi_{B_{i}}T)(x_{j})-T(x_{j}))\|_{p}<\delta.
\end{equation}
 For $A\subseteq \{1,\dots,n\},$ we use $1\otimes\chi_{A}$ for the operator on  $l^{\infty}(\NN,l^{p}(n))$ given by
\[[(1\otimes \chi_{A})f](j)=\chi_{A}f(j),\mbox{  }\, f\in l^{\infty}(\NN,l^{p}(n)),j\in \NN.\]
If we now force $\delta<\kappa,$ then by our choice of $\kappa,m,N$ and (\ref{E:YPMF}),
\begin{align*}
\alpha_{S}(\Hom_{\Gamma,l^{p},(A_{i})}(S,F,m,\delta,\sigma_{i}))\subseteq_{2\varepsilon,\rho_{l^{p}(d_{i})}}&(1\otimes \chi_{C_{i}})\alpha_{S}(\Hom_{\Gamma,l^{p}}(S,F,m,\delta,\sigma_{i})\\
&+\{f\in l^{\infty}(\NN,l^{p}(A_{i}^{c})):f(j)=0,\mbox{ if $j>N$}\}.
\end{align*}
Thus,
\[d_{3\varepsilon}(\alpha_{S}(\Hom_{\Gamma,l^{p},(A_{i})}(S,F,m,\delta,\sigma_{i})),\rho_{l^{p}}(d_{i}))\leq N|A_{i}^{c}|+d_{\varepsilon}(\alpha_{S}(\Hom_{\Gamma,l^{p}}(S,F,m,\delta,\sigma_{i})),\rho_{l^{p}(d_{i})}).\]
As
\[\frac{|A_{i}^{c}|}{d_{i}}\to 0,\]
dividing by $d_{i},$ taking the limit supremum over $i,$ then the limit supremum over $(F,m,\delta)$ and letting $\varepsilon\to 0$ proves that
\[\dim_{\Sigma,l^{p}}(S,\Gamma,(A_{i}),\rho)\leq \dim_{\Sigma,l^{p}}(V,\Gamma).\]

\end{proof}

It may not seem clear to the reader that amenability is important in the previous proposition. Note that $T\in \Hom_{\Gamma,l^{p},(A_{i})}(S,F,m,\delta,\sigma_{i})$ means that $\|T\|\leq 1$ and for all $1\leq k\leq m$ and $s_{1},\dots,s_{k}\in F$ we have
\[\|T(s_{1}\cdots s_{k}a_{j})-\sigma_{i}(s_{1})\cdots \sigma_{i}(s_{k})T(a_{j})\|_{l^{p}(A_{i})}<\delta,\]
whereas  $T\in \Hom_{\Gamma,l^{p}}(S,F,m,\delta,\sigma_{i})$ means that $\|T\|\leq 1$ and for all $1\leq k\leq m,$ and $s_{1},\dots,s_{k}\in F$ we have
\[\|T(s_{1}\cdots s_{k}a_{j})-\sigma_{i}(s_{1})\cdots \sigma_{i}(s_{k})T(a_{j})\|_{p}<\delta.\]
We invite the reader to think about the fact that the first does not imply the second, even if we replace $\delta$ with $\delta'$ for some $\delta'$ depending on $\delta$ with $\delta'\to 0$ as $\delta\to 0.$ Part of the essential difficulty is that our assumptions do not imply anything about $\|T(a_{j})\|_{l^{p}(A_{i}^{c})}$ even if we know $A_{i}^{c}$ is ``small.'' For example, we could a priori have that $T(a_{j})$ is supported on a small set. The significance of amenability is that we can use \emph{any} linear map
\[T\colon V_{F,m}\to l^{p}(d_{i}),\]
and, by using the operators $P_{i}^{(E)}$ in the proposition, produce an almost equivariant map. Moreover if we start with an element of $\Hom_{\Gamma,l^{p},(A_{i})}(S,F,m,\delta,\sigma_{i}),$ then $P_{i}^{(E)}(T)$ will be close to $T$ (after projecting to a complement of a small dimensional space).

\begin{cor}\label{C:soficindependenceofLP} Let $\Gamma$ be an amenable group with a uniformly bounded action on a separable Banach space $V.$ Let $\Sigma=(\sigma_{i}\colon \Gamma\to S_{d_{i}}),$ $\Sigma'=(\sigma_{i}'\colon \Gamma\to S_{d_{i}})$ be two sofic approximations. Then for all $1\leq p\leq \infty,$
\[\dim_{\Sigma,l^{p}}(V,\Gamma)=\dim_{\Sigma',l^{p}}(V,\Gamma).\]
\end{cor}
\begin{proof} An ultrafilter argument using Theorem 1 of \cite{ESZ2}  shows that we can find $\tau_{i}\colon S_{d_{i}}\to S_{d_{i}}$ such that
\[d_{\Hamm}(\tau_{i}\sigma_{i}(s)\tau_{i}^{-1},\sigma_{i}'(s))\to 0.\]
Replacing $\sigma_{i}$ by $\tau_{i}\circ \sigma_{i}\circ \tau_{i}^{-1},$ we may assume that
\[d_{\Hamm}(\sigma_{i}(s),\sigma_{i}'(s))\to 0\]
for all $s\in \Gamma.$ In this case, we can find $A_{i}\subseteq \{1,\dots,d_{i}\}$ such that
\[\frac{|A_{i}|}{d_{i}}\to 1\]
and for all $s_{1},\dots,s_{n}\in \Gamma,$ we have
\[\sigma_{i}(s_{1}\cdots s_{n})(j)=\sigma_{i}(s_{1})\cdots \sigma_{i}(s_{n})(j)=\sigma_{i}'(s_{1})\cdots \sigma_{i}'(s_{n})(j)=\sigma_{i}'(s_{1}\cdots s_{n})(j)\]
 for all $j\in A_{i}$ and all sufficiently large $i.$ Thus if we are given a finite $F\subseteq \Gamma,$ an $m\in \NN$ and a $\delta>0,$ then for all large $i$
\[\Hom_{\Gamma,l^{p},(A_{i})}(S,F,m,\delta,\sigma_{i})=\Hom_{\Gamma,l^{p},(A_{i})}(S,F,m,\delta,\sigma_{i}').\]
The corollary now follows from the preceding proposition.

\end{proof}

\begin{proposition}\label{P:finitegroup} Let $\Gamma$ be a finite group acting on a finite-dimensional vector space $V.$ For $n\in \NN,$ let
\[n=q_{n}|\Gamma|+r_{n}\]
where $0\leq r_{n}<|\Gamma|$ and $q_{n},r_{n}\in \NN.$ Let $A_{n}$ be a set of size $r_{n}$ and define a sofic approximation  $\Sigma=(\sigma_{n}\colon \Gamma\to \Sym((\Gamma\times \{1,\dots,q_{n}\})\sqcup A_{n})$ by
\[\sigma_{n}(s)(g,j)=(sg,j)\mbox{ for $s\in \Gamma,1\leq j\leq q_{n},$}\]
\[\sigma_{n}(s)(a)=a\mbox{ for $a\in A_{n}$}.\]
Then for any $1\leq p\leq \infty$
\[\dim_{\Sigma,l^{p}}(V,\Gamma)=\underline{\dim}_{\Sigma,l^{p}}(V,\Gamma)=\frac{\dim_{\CC}V}{|\Gamma|}.\]

\end{proposition}

\begin{proof} Fix a norm on $V.$ By finite dimensionality, we may use the operator norm on $B(V,l^{p}(d_{i}))$ as our pseudonorm and we may replace $\Hom_{\Gamma}(S,\Gamma,m,\delta,\sigma_{i})$ by the space $\Hom'_{\Gamma}(\Gamma,m,\delta,\sigma_{i})$ of all operators $T\colon V\to l^{p}(d_{i})$ such that
\[\|T\circ s_{1}\cdots s_{k}-\sigma_{i}(s_{1})\cdots \sigma_{i}(s_{k})\circ T\|<\delta\]
for all $1\leq k\leq m,s_{1},\dots,s_{k}\in \Gamma.$

 Let $V_{n}\subseteq B(V,l^{p}(n))$ be the linear subspace of all linear operators
\[T\colon V\to l^{p}(\Gamma\times \{1,\dots,q_{n}\})\]
which are equivariant with respect to the $\Gamma$-action on $l^{p}(\Gamma\times \{1,\dots, q_{n}\})$ given by
\[(gf)(x,j)=f(g^{-1}x,j)\mbox{ for $g,x\in\Gamma,1\leq j\leq q_{n}$.}\]
 Note that we have norm one projections
\[B(V,l^{p}(n))\to B(V,l^{p}(\Gamma\times \{1,\dots,q_{n}\}),\]
\[B(V,l^{p}(\Gamma\times \{1,\dots,q_{n}\})\to V_{n},\]
given by multiplication by $\chi_{\{1,\dots,q_{n}\}}$ and by
\[T\to \frac{1}{|\Gamma|}\sum_{s\in \Gamma}\sigma_{n}(s)^{-1}\circ T\circ s.\]
Let $P_{n}$ denote the composition of these two projections. Since we have a norm one projection from $B(V,l^{p}(n))\to V_{n}$ the Riesz Lemma implies that
\begin{equation}\label{E:RieszLemma}
d_{\varepsilon}(\{T\in V_{n}:\|T\|\leq 1\},\|\cdot\|)\geq \dim_{\CC} V_{n}.
\end{equation}
With the norm in (\ref{E:RieszLemma}) being the operator norm. Define an action of $\Gamma$ on $V^{*}$ by $(g\phi)(x)=\phi(g^{-1}x).$ Let $W_{n}$ be the set of all $\Gamma$-equivariant operators in $B(l^{p}(\Gamma\times \{1,\dots,q_{n}\},V^{*}).$ We use $T^{t}$ for the Banach space adjoint of $T.$  Then $T\mapsto T^{t}$ defines an isomorphism $V_{n}\cong W_{n}.$ For $f\in l^{p}(\Gamma),k\in l^{p}(\{1,\dots,q_{n}\})$ let $f\otimes k$ be defined  by $(f\otimes k)(g,j)=f(g)k(j).$ We leave it as an exercise to the reader to verify that the map
\[\Phi\colon W_{n}\to B(l^{p}(\{1,\dots,q_{n}),V^{*})\]
given by
\[\Phi(T)(f)=T(\chi_{\{e\}}\otimes f)\]
is an isomorphism. Thus,
\[\dim_{\CC}(V_{n})=\dim_{\CC}(W_{n})=q_{n}\dim_{\CC}(V).\]
For $T\in \Hom'_{\Gamma}(\Gamma,m,\delta,\sigma_{i})$ we have
\[\|P_{n}(T)-T\|_{B(V,l^{p}(n))}<\delta.\]
Thus
\begin{equation}\label{E:projectionestimate}
d_{\varepsilon}(\Hom'_{\Gamma}(\Gamma,m,\delta,\sigma_{i}),\|\cdot\|)\leq (\dim_{\CC}V)q_{n}+r_{n},
\end{equation}
and $(\ref{E:RieszLemma}),(\ref{E:projectionestimate})$ are enough to imply the proposition.

\end{proof}

\begin{cor}\label{C:finitegroup} Let $\Gamma$ be a finite group acting on a finite-dimensional vector space $V.$ For any sofic approximation $\Sigma=(\sigma_{i}\colon \Gamma\to S_{d_{i}})$ of $\Gamma$ and $1\leq p\leq \infty$ we have
\[\dim_{\Sigma,l^{p}}(V,\Gamma)=\underline{\dim}_{\Sigma,l^{p}}(V,\Gamma)=\frac{\dim_{\CC}V}{|\Gamma|}.\]

\end{cor}

\begin{proof} Take
\[\Sigma'=(\rho_{d_{i}}\colon \Gamma\to S_{d_{i}})\]
where $\rho_{n}$ is defined as in  Proposition \ref{P:finitegroup}, then use Proposition \ref{P:finitegroup} and Corollary \ref{C:soficindependenceofLP}.

\end{proof}

\begin{proposition}\label{P:Zcase} Let $V$ be a finite-dimensional Banach space with a uniformly bounded action of $\ZZ.$ Let $\sigma_{n}\colon \ZZ\to \Sym(\ZZ/n\ZZ)$ be given by the quotient map $\ZZ\to \ZZ/n\ZZ.$ Then for all $1\leq p\leq \infty,$
\[\dim_{\Sigma,l^{p}}(V,\ZZ)=0.\]
\end{proposition}
\begin{proof}

	Since all norms on a finite-dimensional space are equivalent, we may assume that $V$ is a Hilbert space. Since $V$ is now a Hilbert space, we will call it $H$ instead. Let $\pi\colon \ZZ\to B(H)$ be the representation given by the action of $\ZZ,$ and let $K=\overline{\pi(\ZZ)}.$ By finite-dimensionality, $K$ is a compact group. Let $\ip{\cdot,\cdot}_{H}$ be the inner product on $H.$ Define a new inner product on $H$ by
\[\ip{\xi,\eta}=\int_{K}\ip{T\xi,T\eta}_{H}\,dT,\]
where the integration is with respect to the Haar measure on $K.$ We leave it as an exercise to verify that this is indeed an inner product inducing a norm equivalent to the original norm on $H$ and that $K$ acts unitarily with respect to $\ip{\cdot,\cdot}.$ Thus we may assume that $\pi(\ZZ)\subseteq U(H).$ Set $U=\pi(1).$ By passing to direct sums, we may assume that $\pi$ is irreducible. Thus if we fix any $\xi\in H$ with $\|\xi\|=1,$ then $\xi$ is generating. We will take $S=(\xi,0,0,\dots)$ and as a pseudonorm we take
\[\rho(T)=\|T(\xi)\|.\]
Fix $n\in \NN,$ since $\alpha_{S}(T)(j)=0$ for all $T\in B(\mathcal{H},l^{p}(d_{i}))$ and  $j\geq 2,$ we may view $\alpha_{S}$ as a map into $l^{p}(d_{i}).$

	Fix $1>\varepsilon>0,$ and let $\varepsilon>\delta>0$. Choose $k$ such that $\delta^{p} k<\varepsilon,$ (if $p=\infty$ then let $k$ be any integer.) Since $\overline{\pi(\ZZ)}$ is compact, we can find an integer $m$ such that
\[\|U^{mj}-1\|<\delta,\]
for $1\leq j\leq k.$
We may assume that $m$ is large enough so that $\{U^{j}\xi:-m\leq j\leq -1\}$ spans $H.$  Let $F=\{j\in \ZZ:|j|\leq m(2k+1)\}$. For $n\in\NN,$ let $q_{n}\in \NN\cup\{0\},0\leq r_{n}<k$ be the integers defined by
\[n=q_{n}mk+r_{n}.\]
Define $Q_{j},j=0,\dots,k-1$ by
\[Q_{j}=\bigcup_{l=1}^{m}\{jm+l+qm_{k}:0\leq q\leq q_{n}-1\}.\]
Pictorially, if we think of $\{1,\dots, q_{n}mk\}$ as a rectangle formed out of $mk$ horizontal dots and $q_{n}$ vertical dots, then $Q_{j}$ is the rectangle from the $jm+1^{st}$ horizontal dot to the $(j+1)m^{th}$ horizontal dot.  Fix $T\in \Hom_{\Gamma}(S,F,m,\delta,\sigma_{n}).$   Let $f_{j}\colon Q_{j}\to \CC$ be given by
\[f_{j}(l)=T(\xi)(\sigma_{n}(mj)^{-1}(l)).\]
We then extend $f_{j}$ to a map $\ZZ/n\ZZ\to\CC$ by saying that $f_{j}(r)=0$ for $r\in\ZZ/n\ZZ\setminus Q_{j}.$ Note that for $1\leq p<\infty,$
\begin{align*}\left\|T(\xi)-\sum_{j=0}^{k-1}f_{j}\right\|_{l^{p}(\{1,\dots,q_{n}mk\})}&=\left(\sum_{j=0}^{k-1}\|T(\xi)-\sigma_{n}(mj)T(\xi)\|_{l^{p}(Q_{j})}^{p}\right)^{1/p}\\
&<\delta k+\left(\sum_{j=0}^{k-1}\|T((U^{-mj}-1)\xi)\|_{p}^{p}\right)^{1/p}\\
&<2\delta k\\
&<2\varepsilon,
\end{align*}
using the triangle inequality on $l^{p}.$ Similarly for $p=\infty,$
\[\left\|T(\xi)-\sum_{j=0}^{k-1}f_{j}\right\|_{l^{\infty}(\{1,\dots,q_{n}mk\})}<2\varepsilon.\]
Additionally, if $l\in Q_{0}$ and  $0\leq s\leq k-1,$ then
\[f(l+ms)=\sum_{j=0}^{k-1}f_{j}(l+ms),\]
and since $Q_{0},\cdots, Q_{k-1}$ are pairwise disjoint the only nonzero term in the above sum is when $j=s.$ Thus
\[f(l+ms)=f_{s}(l+ms)=T(\xi)(l).\]
So $f$ is constant on $\{l+mt:0\leq t\leq k-1\}.$ Thus
\begin{align*}
\alpha_{S}(\Hom_{\Gamma}(S,F,m,&\delta,\sigma_{n}))\subseteq_{2\varepsilon,\|\cdot\|_{p}} l^{p}(\{1,\dots,n\}\setminus\{1,\dots,q_{n}mk\})+\\
&\{f\in l^{p}(q_{n}mk):f(i+mj)=f(i),\mbox{ for all } i\in Q_{0},0\leq j\leq k-1\}.
\end{align*}
So
\[\frac{1}{n}d_{2\varepsilon}(\alpha_{S}(\Hom_{\Gamma}(S,F,m,\delta,\sigma_{n}),\|\cdot\|_{p})\leq \frac{q_{n}m}{n}+\frac{r_{n}}{n}.\]
Letting $n\to \infty,$ taking the limit supremum over $(F,m,\delta)$ and then letting $\varepsilon\to 0$ we conclude that
\[\dim_{\Sigma,l^{p}}(\mathcal{H},\Gamma)\leq\frac{1}{k}.\]
Since $k$ becomes arbitrarily large when $\delta$ becomes small (or can be made arbitrarily large when $p=\infty$), this completes the proof.

\end{proof}

	We will now proceed to prove that if $\Gamma$ is an infinite sofic group and $\Sigma$ is a sofic approximation of $\Gamma,$ then for any finite-dimensional representation $V$ of $\Gamma$ we have
\[\dim_{\Sigma,l^{p}}(V,\Gamma)=0.\]
The method is based on passing to an action of the group on a measure space and then using that the corresponding equivalence relations contains an action of $\ZZ.$

	We shall first work with the trivial action of $\Gamma$ on $\CC.$ For this, fix a sofic group $\Gamma$ and a sofic approximation $\Sigma$ . For $S=(1,0,0,\dots),$ and the trivial action of $\Gamma$ on $\CC$ the map $T\to T(\{1\})$ identifies $\Hom_{\Gamma,p}(S,F,m,\delta,\sigma_{i})$ with all vectors $\xi\in l^{p}(d_{i})$ such that
\[\|\sigma_{i}(g)\xi-\xi\|_{p}<\delta\]
for all $g\in F.$ For the proof of the next Lemma we will also need the concept of a sofic approximation of an equivalence relation. Let us recall some preliminary definitions.

\begin{definition} \emph{A discrete equivalence relation, is a triple $(\R,X,\mu)$ where $X$  is a standard Borel space, $\mu$ is a Borel probability measure on $X,$  $\R\subseteq X\times X$ is a  Borel subset such that the relation $x\thicksim y$ if $(x,y)\in \R$ is an equivalence relation and such that for almost every $x\in X, \{y:(x,y)\in \R\}$ is countable. We say that $\R$ is} measure-preserving \emph{if for all Borel $A\subseteq \R$}
\[\int_{X}|\{y\in X:(x,y)\in A\}|\,d\mu(x)=\int_{X}|\{x\in X:(x,y)\in A\}|\,d\mu(y).\]
\emph{We shall denote the above quantity by $\overline{\mu}(A)$. Then $\overline{\mu}$ is a measure on the Borel subsets of $\R.$ We say that $\R$ is} ergodic \emph{if for all measurable $f\colon X\to \CC$  with $f(x)=f(y)$ for $\overline{\mu}$-almost every $(x,y)\in \R,$ there is a $\lambda\in \CC$ with $f(x)=\lambda$ almost everywhere with respect to $\mu$.}
\end{definition}

	The main example of relevance for us is given by taking a countable discrete group $\Gamma$ and a free measure-preserving action $\Gamma\actson (X,\mu)$ where $(X,\mu)$ is a standard probability space. In this case $\R_{\Gamma\actson(X,\mu)}=\{(x,gx):g\in \Gamma\}.$  Such an action is $\emph{free}$ if for all $g\in \Gamma\setminus\{e\},$ $\mu(\{x\in X:gx=x\})=0.$

\begin{definition}\emph{Let $(\R,X,\mu)$ be a discrete, measure-preserving equivalence relation. A} partial morphism \emph{is a bimeasurable bijection  $\phi\colon \dom(\phi)\to \ran(\phi)$ where $\dom(\phi),\ran(\phi)\subseteq X$ are measurable and $(x,\phi(x))\in \R$ for almost every $x\in \dom(\phi).$ We let $\phi^{-1}$ be the partial morphism with $\dom(\phi^{-1})=\ran(\phi),\ran(\phi^{-1})=\dom(\phi)$ and $\phi(\phi^{-1}(x))=x$ for $x\in \ran(\phi).$ If $\phi,\psi\in [[\R]],$ we let $\phi\circ \psi$ be the partial morphism with $\dom(\phi\circ \psi)=\{x\in \dom(\psi):\psi(x)\in \dom(\phi)\},$ and $(\phi\circ \psi)(x)=\phi(\psi(x)).$ If $A\subseteq X$ is measurable, we let $\id_{A}\colon A\to A$ be the partial morphism which is the identity on $A.$  If $\phi,,\psi$ and partial morphisms with $\mu(\dom(\phi)\cap \dom(\psi))=0,\mu(\ran(\phi)\cap \ran(\psi))=0,$ we let $\phi+\psi$ be the partial morphism with domain $\dom(\phi)\cup \dom(\psi)\setminus[\dom(\phi)\cap \dom(\psi)]$ which is equal to $\phi(x)$ for $x\in \dom(\phi)\setminus \dom(\psi),$ and equal to $\psi(x)$ for $x\in \dom(\psi)\setminus [\dom(\phi)\cap \dom(\psi)].$ We leave it as an exercise to the reader to verify that, after removing a null set, $\phi+\psi$ is a partial morphism. We let $[[\R]]$ be the set of all partial morphisms of $\R,$ we let $[\R]$ be the set of all partial morphisms of $\R$ with $\mu(\dom(\phi))=1. $ We identify two elements $\phi,\psi$ of $[[\R]]$ if $\mu(\dom(\phi)\Delta\dom(\psi))=0,$ and $\mu(\{x\in \dom(\phi)\cap \dom(\psi):\phi(x)\ne\psi(x)\})=0.$  If  $\Gamma\actson(X,\mu)$ by measure-preserving transformations then we let  $\alpha_{g}\colon X\to X$ for $g\in \Gamma$ be the element of $[\R_{\Gamma\actson (X,\mu)}]$ defined by $\alpha_{g}(x)=gx.$ We define a metric on $[[\R]]$ by}
\[d_{[[\R]]}(\phi,\psi)^{2}=\mu(\dom(\phi)\Delta\dom(\psi))+2\mu(\{x\in \dom(\phi)\cap \dom(\psi):\psi(x)\ne \phi(x)\}).\]
\end{definition}
	For the next definition, we need to note the following example. For $n\in \NN,$ let $R_{n}$ be the equivalence relation on $\{1,\dots,n\}$ declaring all points to be equivalent. We may view $S_{n}\subseteq \{1,\dots,n\}$ since both are (partially defined) functions on $\{1,\dots,n\},$ indeed $[\R_{n}]=S_{n}.$

\begin{definition} \emph{Let $\Gamma$ be a countable discrete sofic group with sofic approximation $\Sigma=(\sigma_{i}\colon \Gamma\to S_{d_{i}}).$ Let $\Gamma$ have a free, measure-preserving action on a standard probability space $(X,\mu).$ A} sofic approximation of $\R_{\Gamma\actson (X,\mu)}$ extending $\Sigma$\emph{ is a sequence of maps $\Sigma'=(\rho_{i}\colon [[\R]]\to [[\R_{n}]])$ such that}
\[\rho_{i}(u_{g})=\sigma_{i}(g),\mbox{ \emph{for all $g\in \Gamma$}},\]
\[\mbox{\emph{ for all $A\subseteq X$ measurable, there exists $A_{i}\subseteq\{1,\dots,d_{i}\}$ such that $\rho_{i}(\id_{A})=\id_{A_{i}}$}},\]
\[\frac{|\{1\leq j\leq d_{i}:j\in \dom(\rho_{i}(\phi)),\rho_{i}(\phi)(j)=j\}|}{d_{i}}\to \mu(\{x\in \dom(\phi):\phi(x)=x\}),\]
\[d_{[[\R_{d_{i}}]]}(\rho_{i}(\phi^{-1}),\rho_{i}(\phi)^{-1})\to 0,\mbox{\emph{for all $\phi\in [[\R_{n}]]$}},\]
\[d_{[[\R_{d_{i}}]]}(\rho_{i}(\phi\psi),\rho_{i}(\phi)\rho_{i}(\psi))\to 0,\mbox{\emph{for all $\phi,\psi\in [[R_{n}]]$}}.\]
\end{definition}
The notion of a sofic equivalence relation is due to Elek-Lippner in \cite{ElekLip}.

\begin{lemma} Let $\Gamma$ be a countable discrete sofic group with a sofic approximation $\Sigma.$ Let $\Gamma\actson (X,\mu)$ be a free, ergodic, measure-preserving action on a standard probability space $(X,\mu)$ such that there is a sofic approximation (still denoted $\Sigma$) of $\R_{\Gamma\actson (X,\mu)}$ extending the sofic approximation of $\Gamma.$ Let $\Sigma=(\sigma_{i}\colon [[\R]]\to [[\R_{d_{i}}]]).$ Fix $\phi\in [[\R]],$ and $\eta>0.$ Then there are $F\subseteq \Gamma$ finite, $m\in \NN,$ $\delta>0$ and $C_{i}\subseteq \{1,\dots,d_{i}\}$ with $|C_{i}|\geq (1-\eta)d_{i}$ with the following property. For the trivial representation of $\Gamma$ on $\CC,$ and $T\in \Hom_{\Sigma,p}((1,0,0,\dots),F,m,\delta,\sigma_{i})$ with $\xi=T(1)$ we have
\[\|\sigma_{i}(\phi)\xi-\sigma_{i}(\id_{\ran(\phi)})\xi\|_{l^{p}(C_{i})}<\eta,\]
for all large $i.$
\end{lemma}

\begin{proof}

	Let $\{A_{g}:g\in \Gamma\}$ be a partition of $\ran(\phi)$ with
\[\phi=\sum_{g\in \Gamma}\id_{A_{g}}\alpha_{g},\]
with the sum converging $d_{[[\R]]}.$ Choose $F\subseteq \Gamma$ finite such that
\[d_{[[\R]]}\left(\phi,\sum_{g\in F}\id_{A_{g}}\alpha_{g}\right)<\eta.\]
For $\xi\in l^{p}(d_{i}),\psi\in [[R_{d_{i}}]],$ we use
\[(\psi\xi)(j)=\chi_{\ran(\psi)}(j)\xi(\psi^{-1}(j)).\]
For $A\subseteq\{1,\dots,d_{i}\}$ we  use $\chi_{A}$ for the operator on $l^{p}(d_{i})$ given by multiplication by $A.$ By soficity, for all large $i$ we may find a $C_{i}\subseteq\{1,\dots,d_{i}\}$ with $|C_{i}|\geq (1-2\eta)d_{i}$ and
\[\chi_{C_{i}}\sigma_{i}(\phi)=\sum_{g\in F}\chi_{C_{i}}\sigma_{i}(\id_{A_{g}})\sigma_{i}(g),\]
\[\chi_{C_{i}}\sigma_{i}(\id_{\ran(\phi)})=\sum_{g\in F}\chi_{C_{i}}\sigma_{i}(\id_{A_{g}}),\]
as operators on $l^{p}(d_{i}).$ Let $m\in \NN,$ and let $\delta>0$ be sufficently small in a manner to be determined later. For $T,\xi$ as in the statement of the lemma,
\[\chi_{C_{i}}\sigma_{i}(\phi)\xi=\sum_{g\in F}\chi_{C_{i}}\sigma_{i}(\id_{A_{g}})\sigma_{i}(g)\xi,\]
\[\chi_{C_{i}}\sigma_{i}(\id_{\ran(\phi)})\xi=\sum_{g\in F}\chi_{C_{i}}\sigma_{i}(\id_{A_{g}})\xi,\]
so
\[\|\sigma_{i}(\phi)\xi-\sigma_{i}(\id_{\ran(\phi)})\xi\|_{l^{p}(C_{i})}\leq |F|\delta.\]
So if $\delta<\frac{\eta}{|F|},$ our claim is proved.

\end{proof}

\begin{lemma} Let $\Gamma$ be a countably infinite discrete sofic group with sofic approximation $\Sigma.$ Then for the trivial representation of $\Gamma$ on $\CC,$ we have
\[\dim_{\Sigma,l^{p}}(\CC,\Gamma)=0.\]
\end{lemma}

\begin{proof} Let $\R$ be the equivalence relation induced by the Bernoulli action of $\Gamma$ on $(X,\mu)=(\{0,1\},u)^{\Gamma},$ $u$ being the uniform measure. Extend $\Sigma$ to a sofic approximation of $[[\R]],$ (this is essentially possible by \cite{Bow} Theorem 8.1, see also \cite{ElekLip} Proposition 7.1,\cite{DKP} Theorem 5.5, \cite{Pan} Theorem 2.1).  Let $S=(1,0,0,\dots).$ Since $\Gamma$ is an infinite group, by (\cite{KM} Corollary 7.6) we know that for all $n\in \NN$ there is a subequivalence relation $\R_{n}$ of $\R$ generated by a free, measure-preserving action of $\ZZ/n\ZZ$ on $(X,\mu).$ Let $\alpha\in [\R_{n}]$ generate the action of $\ZZ/n\ZZ$ on $(X,\mu).$ Fix $\eta>0.$ By the preceding lemma, we may choose a finite subset $F\subseteq \Gamma,\delta>0$ and subsets $C_{i}\subseteq \{1,\dots,d_{i}\}$ with $|C_{i}|\geq (1-d_{i})\eta$ such that if $T\in \Hom_{\Gamma}(S,F,1,\delta,\sigma_{i})$ and $\xi=T(1),$ then for all large $i$
\[\|\sigma_{i}(\alpha)^{j}\xi-\xi\|_{l^{p}(C_{i})}<\eta,\mbox{for $1\leq j\leq n-1$.}\]
We may assume that there are $A_{i}\subseteq \{1,\dots,d_{i}\}$ with $\frac{|A_{i}|}{d_{i}}\to \frac{1}{n}$ and 
\[\{\sigma_{i}(\alpha)^{j}(A_{i}):0\leq j\leq n-1\} \mbox{ are a disjoint family},\]
\[\sigma_{i}(\alpha)\big|_{\{1,\dots,d_{i}\}\setminus\bigcup_{j=0}^{n-1}\sigma_{i}(\alpha)^{j}(A_{i})}=\id.\]

	Let
\[\zeta=\sum_{j=1}^{n}\sigma_{i}(\alpha)^{j}\chi_{A_{i}}\xi=\sum_{j=1}^{n}\chi_{\sigma_{i}(\alpha)^{j}(A_{i})}\sigma_{i}(\alpha)^{j}\xi.\]
Set $D_{i}=C_{i}\cap \bigcup_{j=0}^{n-1}\sigma_{i}(\alpha)^{j}(A_{i}).$ Then
\[\chi_{D_{i}}\zeta-\chi_{D_{i}}\xi=\sum_{i=1}^{n}\chi_{D_{i}\cap \sigma_{i}(\alpha)^{j}(A_{i})}(\sigma_{i}(\alpha)^{j}\xi-\xi),\]
so
\[\|\chi_{D_{i}}\zeta-\chi_{D_{i}}\xi\|_{p}\leq \eta n.\]
Since $\alpha_{S}(T)(l)=0$ for all $l\in \NN,\l\geq 2$ and $T\in B(\CC,l^{p}(d_{i})),$ we may view $\alpha_{S}$ as a map into $l^{p}(d_{i}).$ Then
\begin{align*}
\alpha_{S}(\Hom_{\Gamma}(S,F,m,\delta,\sigma_{i}))\subseteq_{\eta n,\|\cdot\|_{p}}&\chi_{D_{i}}\left\{\sum_{j=1}^{n}\sigma_{i}(\alpha)^{j}f:f\in l^{p}(A_{i})\right\}\\
&+l^{p}(\{1,\dots,d_{i}\}\setminus D_{i}).
\end{align*}
As
\[\frac{|D_{i}|}{d_{i}}\to 1,\]
\[\frac{|A_{i}|}{d_{i}}\to \frac{1}{n}\]
we find that
\[\limsup_{(F,m,\delta)}\limsup_{i\to \infty}\frac{1}{d_{i}}d_{\eta n}(\alpha_{S}(\Hom_{\Gamma}(S,F,m,\delta,\sigma_{i})),\|\cdot\|_{p})\leq \frac{1}{n}.\]
Letting $\eta\to 0,$ and then $n\to \infty$ completes the proof.

\end{proof}

\begin{theorem} Let $\Gamma$ be a countably infinite sofic group with sofic approximation $\Sigma.$ Then for any representation of $\Gamma$ on a finite-dimensional vector space $V$ and for all $1\leq p<\infty,$
\[\dim_{\Sigma,l^{p}}(V,\Gamma)=0.\]

\end{theorem}

\begin{proof} Let $\Sigma=(\sigma_{i}\colon\Gamma\to S_{d_{i}}).$ We first prove that the Theorem is true when $\Gamma=\ZZ.$  For this, consider the sofic approximation
\[\sigma'_{i}\colon \ZZ\to \Sym(\ZZ/n\ZZ))\]
defined as in Proposition \ref{P:Zcase}. Consider $\Sigma'=(\sigma'_{d_{i}}\colon \ZZ\to \Sym(\ZZ/d_{i}\ZZ)).$ Applying Proposition \ref{P:Zcase} and Corollary \ref{C:soficindependenceofLP},
\[\dim_{\Sigma,l^{p}}(V,\ZZ)=\dim_{\Sigma',l^{p}}(V,\ZZ)\leq 0.\]
 This prove the case $\Gamma=\ZZ.$

Similary Corollary \ref{C:finitegroup}  implies that if $\Lambda$ is a finite subgroup of $\Gamma$ then
\[\dim_{\Sigma,l^{p}}(V,\Gamma)\leq \dim_{\Sigma,l^{p}}(V,\Lambda)=\frac{\dim_{\CC}(V)}{|\Lambda|}.\]
Since dimension increases under restricting the action to a subgroup, we see that the Theorem holds if
\[\{|\Lambda|:\Lambda \mbox{ is a finite subgroup of } \Gamma\}\]
is unbounbed.

	Again because dimension increases under passing to subgroups, we may assume that
\[\{|\Lambda|:\Lambda \mbox{ is a finite subgroup of } \Gamma\}\]
is bounded and that every element of $\Gamma$ has finite order. As in Proposition \ref{P:Zcase} we may assume that $V$ is a Hilbert space and $\Gamma$ acts by unitaries. Let $M$ be greater than $|\Lambda|$ for any finite subgroup of $\Gamma.$ Choose $\varepsilon>0$ such that if $U$ is a unitary on a Hilbert space and
\[\|U-1\|<\varepsilon,\]
then $U^{M}\ne 1$ unless $U=1.$ Let $\pi\colon \Gamma\to U(V)$ be the homomorphism induced by the action of $\Gamma.$  By finite-dimensionality $\overline{\pi(\Gamma)}$ is compact, so we may find an infinite sequence $(g_{n})_{n=1}^{\infty}$ of distinct elements of $\Gamma$ with
\[\|\pi(g_{n})-1\|<\varepsilon.\]
If
\[\Lambda=\ip{g_{n}:n\in \NN},\]
our assumptions then imply that $\Lambda$ is an infinite subgroup of $\Gamma$ which acts trivially. Thus by the preceding lemma and subadditivity of dimension under exact sequences,
\[\dim_{\Sigma,l^{p}}(V,\Gamma)\leq \dim_{\Sigma,l^{p}}(V,\Lambda)=0.\]

\end{proof}

\section{A Complete Calculation In The Case of $\bigoplus_{j=1}^{n}L^{p}(L(\Gamma))q_{j}).$}\label{S:noncommLp}

	In this section, we show that if $\Gamma$ is $\R^{\omega}$-embeddable, $\Sigma$ is an embedding sequence and $q_{1},\dots,q_{n}\in \Proj(L(\Gamma)),$ then
\[\dim_{\Sigma,S^{p},\textnormal{mult}}\left(\bigoplus_{j=1}^{n}L^{p}(L(\Gamma),\tau)q_{j},\Gamma\right)=\underline{\dim}_{\Sigma,S^{p},\textnormal{mult}}\left(\bigoplus_{j=1}^{n}L^{p}(L(\Gamma),\tau)q_{j},\Gamma\right)=\]
\[\sum_{j=1}^{n}\tau(q_{j})\]
where $\tau$ is the group trace.

		We shall frequently use functional calculus throughout the proofs. For notation, if $T$ is a closed, densely-defined operator on a Hilbert space $H,$ then $|T|=(T^{*}T)^{1/2}.$ We use $\{u_{g}:g\in \Gamma\}$ for the canonical unitaries generating $L(\Gamma).$ We use $\tr$ for the linear functional on $M_{n}(\CC)$ equal to $\frac{1}{n}\Tr,$ and for $A\in M_{n}(\CC),$ we use $\|A\|_{p}^{p}=\tr(|A|^{p}).$ As before $\|A\|_{\infty}$ is defined to be the operator norm. We introduce some background. For proofs of the following facts see \cite{TakesakiII} Chapter IX.2.

\begin{definition}\emph{ A subalgebra $M$ of $B(H)$ is a} von Neumann algebra \emph{ if it closed under adjoints, limits in the weak operator topology and contains the identity of $H.$ A} faithful normal tracial state on $M$ \emph{is a linear functional $\tau\colon M\to \CC$ with $\tau(xy)=\tau(yx),$ for $x,y\in M,$ $\tau(x^{*}x)\geq 0$ for all $x\in M,$ with equality if and only if $x=0$ and such that $\tau\big|_{\{x \in M:\|x\|_{\infty}\leq 1\}}$ is weak operator topology continuous. Here, as always, $\|x\|_{\infty}$ is the operator norm of $x.$}\end{definition}
	
	Let $M\subseteq B(H)$ is von Neumann algebra and $\tau\colon M\to \CC$ is a faithful normal tracial state. We say that a closed densely-defined operator $x$ on $H$ is affiliaed to $M$ if it commutes with all the unitaries in 
\[M'=\{T\in B(H):[T,x]=0\mbox{ for all $a\in M$}\}.\]
 For $1\leq p<\infty$ we define
\[L^{p}(M,\tau)\]
to be all closed densely-defined operators $x$ on $H$ affiliated to $M$  such that if
\[|x|=\int_{[0,\infty)}t\,dE(t)\]
is the spectral decomposition of $x,$ then
\[\|x\|_{p}^{p}=\int_{[0,\infty)}t^{p}\,d\tau\circ E(t)<\infty.\]
Furthermore, $\|x\|_{p}$ is a norm on $L^{p}(M,\tau)$ which makes $L^{p}(M,\tau)$ into a Banach space, with the sum being the closure of the operator $x+y$ with domain $\dom(x)\cap \dom(y).$  We have the inequalities
\[\|xy\|_{r}\leq \|x\|_{p}\|y\|_{q}\]
if
\[\frac{1}{r}=\frac{1}{p}+\frac{1}{q}.\]
Here we are using $xy$ for the closure of the (densely-defined) operator with domain $y^{-1}(\dom(x))$ and defined by $xy(\xi)=x(y(\xi)).$

For any closed densely-defined operator $T$ on $H$ we will use $\dom(T),\ran(T)$ for its domain and range.
\begin{lemma}

(a) Let $n\in \NN.$ Suppose that $A,B\in M_{n}(\CC)$ are such that $|A|\leq |B|,$ then for all $\beta>0$
\[\tr(|A|^{\beta})\leq \tr(|B|^{\beta}).\]

(b) Suppose that $A,B\in M_{n}(\CC)$ and $Q$ is an orthogonal projection in $M_{n}(\CC).$ Fix $1\leq p<\infty.$ Suppose that $\delta,\eta>0$ are such that
\[\|(A-1)B\|_{p}<\delta,\|A-Q\|_{p}<\eta.\]
	Then
\[\|B-\chi_{(0,\sqrt{\delta})}(|A-1|)B\|_{p}<\sqrt{\delta},\]
and
\[\tr(\chi_{(0,\sqrt{\delta})}(|A-1|))\leq \tr(Q)+\left(\frac{\eta}{1-\sqrt{\delta}}\right)^{p}.\]

\end{lemma}

\begin{proof} We first make the following preliminary observation: if $P,Q$ are orthogonal projections in $M_{n}(\CC)$ with
\[P\CC^{n}\cap Q\CC^{n}=\{0\},\]
then
\[\tr(P)\leq 1-\tr(Q).\]
This follows directly from the fact that $1-Q$ is injective on $P\CC^{n}.$

(a) First note that
\[\tr(T^{\alpha})=\alpha\int_{0}^{\infty}t^{\alpha-1}\tr(\chi_{(t,\infty)}(T))\,dt\]
if $T\geq 0.$ Suppose that $0\leq T\leq S$ and
\[\xi\in \chi_{(t,\infty)}(T)(\CC^{n})\cap \chi_{[0,t]}(S)(\CC^{n})\]
with $\xi\ne 0,$ then
\[t\|\xi\|^{2}<\ip{T\xi,\xi}\leq \ip{S\xi,\xi}\leq t\|\xi\|^{2},\]
which is a contradiction. Hence
\[\chi_{(t,\infty)}(T)(\CC^{n})\cap \chi_{[0,t]}(S)(\CC^{n})=\{0\},\]
so the above integral formula and our preliminary observation prove $(a).$

(b) Note that
\begin{align*}
|\chi_{[\sqrt{\delta},\infty)}(|A-1|)B|^{2}&=B^{*}\chi_{[\sqrt{\delta},\infty)}(|A-1|)B\\
&\leq \frac{1}{\delta}B^{*}|A-1|^{2}B=\left|\frac{1}{\sqrt{\delta}}(A-1)B\right|^{2}.
\end{align*}
So by $(a)$
\[\|B-\chi_{(0,\sqrt{\delta})}(|A-1|)B\|_{p}=\|\chi_{[\sqrt{\delta},\infty)}(|A-1|)B\|_{p}<\sqrt{\delta}.\]
Further if
\[\xi \in \chi_{(0,\sqrt{\delta})}(|A-1|)(\CC^{n})\cap (1-Q)(\CC^{n})\cap \chi_{[0,1-\sqrt{\delta}]}(|A-Q|)(\CC^{n})\]
is nonzero, then
\[(1-\sqrt{\delta})^{2}\|\xi\|^{2}\geq \ip{|A-Q|^{2}\xi,\xi}=\|A\xi\|^{2}>(1-\sqrt{\delta})^{2}\|\xi\|^{2},\]
which is a contradiction. Thus
\[\tr(\chi_{(0,\sqrt{\delta})}(|A-1|))\leq \tr(Q)+\tr(\chi_{(1-\sqrt{\delta},\infty)}(|A-Q|)).\]
Since
\[\chi_{(1-\sqrt{\delta},\infty)}(|A-Q|)\leq \frac{|A-Q|^{p}}{(1-\sqrt{\delta})^{p}},\]
we have that
\[\tr(\chi_{(1-\sqrt{\delta},\infty)}(|A-Q|))<\frac{\eta^{p}}{(1-\sqrt{\delta})^{p}}.\]

\end{proof}

\begin{proposition}\label{P:LPupbound} Let $\Gamma$ be an $\R^{\omega}$-embeddable group and $\Sigma$ an embedding sequence. Let $M=L(\Gamma)$ and $\tau$ the canonical group trace on $M.$ Then for all $1\leq p<\infty$ and for every $q_{1},\dots,q_{n}\in \Proj(M)$ we have
\[\dim_{\Sigma,S^{p},\mult}\left(\bigoplus_{j=1}^{n}L^{p}(M,\tau)q_{j},\Gamma\right)\leq \sum_{j=1}^{n}\tau(q_{j}).\]
\end{proposition}

\begin{proof} Let $\Sigma=(\sigma_{i}\colon\Gamma\to U(d_{i})).$ By subadditivity of dimension it is enough to show that if $q$ is a projection in $M,$ then
\[\dim_{\Sigma,S^{p},\mult}(L^{p}(M,\tau)q)=\tau(q).\]
Let $0<\varepsilon,\kappa<1/2.$ By \cite{Me} Lemma 5.5, we may extend $\sigma_{i}$ to (potentially nonlinear and nonmultiplicative) maps $\sigma_{i}\colon L(\Gamma)\to M_{d_{i}}(\CC)$ with
\[\sup_{i}\|\sigma_{i}(x)\|_{\infty}<\infty,\mbox{ for all $x\in L(\Gamma)$},\]
\[\tr(\sigma_{i}(x))\to \tau(x),\mbox{ for all $x\in L(\Gamma),$}\]
\[\|P(\sigma_{i}(x_{1}),\dots,\sigma_{i}(x_{n}))-\sigma_{i}(P(x_{1},\dots,x_{n}))\|_{2}\to 0,\]
 for all $x_{1},\dots,x_{n}\in L(\Gamma)$ and all $*$-polynomials in $n$-noncommuting variables.
	
	Let $p\in L(\Gamma)$ be any orthogonal projection. Then
\begin{equation}\label{E:truetautologyistrue}
\|\sigma_{i}(p)-\sigma_{i}(p)^{*}\sigma_{i}(p)\|_{2}\to 0
\end{equation}
\begin{equation}\label{E:truetautologyistrue2}
\|\sigma_{i}(p)^{*}\sigma_{i}(p)-(\sigma_{i}(p)^{*}\sigma_{i}(p))^{2}\|_{2}\to 0.
\end{equation}
By the triangle inequality and functional calculus,
\begin{align}\label{E:SDOLGHSALHGLSAHG}
\|\chi_{[1-\varepsilon,1+\varepsilon]}(\sigma_{i}(p)^{*}\sigma_{i}(p))-\sigma_{i}(p)^{*}\sigma_{i}(p)\|_{2}&\leq\|\chi_{[0,\infty)\setminus [1-\varepsilon,1+\varepsilon]}(\sigma_{i}(p)^{*}\sigma_{i}(p))\sigma_{i}(p)^{*}\sigma_{i}(p)\|_{2}\\ \nonumber
&+\|\chi_{[1-\varepsilon,1+\varepsilon]}(\sigma_{i}(p)^{*}\sigma_{i}(p))(1-\sigma_{i}(p)^{*}\sigma_{i}(p))\|_{2}.\\ \nonumber
\end{align}
For $1>\varepsilon>0,$ we have
\begin{align}\label{E:EasyInequality1}
\|\chi_{[0,\infty)\setminus [1-\varepsilon,1+\varepsilon]}(\sigma_{i}(p)^{*}\sigma_{i}(p))\sigma_{i}(p)^{*}\sigma_{i}(p)\|_{2}^{2}&=\tr(\chi_{[0,\infty)\setminus [1-\varepsilon,1+\varepsilon]}(\sigma_{i}(p)^{*}\sigma_{i}(p))(\sigma_{i}(p)^{*}\sigma_{i}(p))^{2})\\ \nonumber
&\leq \frac{1}{\varepsilon^{2}} \tr(|\sigma_{i}(p)^{*}\sigma_{i}(p)-(\sigma_{i}(p)^{*}\sigma_{i}(p))|^{2})\\ \nonumber
&=\frac{1}{\varepsilon^{2}}\|\sigma_{i}(p)^{*}\sigma_{i}(p)-(\sigma_{i}(p)^{*}\sigma_{i}(p))^{2}\|_{2}^{2}. \nonumber
\end{align}
Here we use functional calculus and the fact that if $t\in \RR$ and $1>\varepsilon>0,$ then
\[\chi_{[0,\infty)\setminus[1-\varepsilon,1+\varepsilon]}(t)t^{2}\leq \frac{1}{\varepsilon^{2}}|t-t^{2}|^{2}.\]
If $0<\varepsilon<1,$ we have
\begin{align}\label{E:EasyInequality2}
\|\chi_{[1-\varepsilon,1+\varepsilon]}(\sigma_{i}(p)^{*}\sigma_{i}(p))(1-\sigma_{i}(p)^{*}\sigma_{i}(p))\|_{2}^{2}&=\tr(\chi_{[1-\varepsilon,1+\varepsilon]}(\sigma_{i}(p)^{*}\sigma_{i}(p))(1-\sigma_{i}(p)^{*}\sigma_{i}(p))^{2})\\ \nonumber
&\leq \frac{1}{(1-\varepsilon)^{2}}\tr(|\sigma_{i}(p)^{*}\sigma_{i}(p)-(\sigma_{i}(p)^{*}\sigma_{i}(p))^{2}|^{2})\\ \nonumber
&=\frac{1}{(1-\varepsilon)^{2}}\|\sigma_{i}(p)^{*}\sigma_{i}(p))-(\sigma_{i}(p)^{*}\sigma_{i}(p))^{2}\|_{2}^{2}.
\end{align}
Here we use functional caclulus and the fact that if $t\in \RR,$ then
\[\chi_{[1-\varepsilon,1+\varepsilon]}(t)(1-t)^{2}\leq \frac{1}{(1-\varepsilon)^{2}}|t-t^{2}|^{2}.\]
Combining (\ref{E:truetautologyistrue}),(\ref{E:truetautologyistrue2}), (\ref{E:SDOLGHSALHGLSAHG}),(\ref{E:EasyInequality1}),(\ref{E:EasyInequality2}) we see that for all  $0<\varepsilon<1,$
\[\|\sigma_{i}(p)-\chi_{[1-\varepsilon,1+\varepsilon]}(\sigma_{i}(p)^{*}\sigma_{i}(p))\|_{2}\to 0.\]
 Applying the above estimates with $p=q,$ we see that we may replacing $\sigma_{i}(q)$ with $\chi_{[3/4,5/4]}(\sigma_{i}(q)^{*}\sigma_{i}(q)).$ Thus, we may assume that $\sigma_{i}(q)$ is an orthogonal projection for all $i.$

	Choose $f\in c_{c}(\Gamma)$ with
\[\left\|q-\sum_{s\in \Gamma}f(s)u_{s}\right\|_{p}<\kappa.\]
If $T\colon L^{p}(M,\tau)q\to L^{p}(M_{d_{i}}(\CC),\tr),$ define
\[\widetidle{T}(x)=T(xq).\]
Let $F$ be the support of $f,$ then if $m\in \NN,\kappa,\delta>0$ are sufficiently small  we have
\[\left\|\left(\sum_{s\in \Gamma}f(s)\sigma_{i}(s)-1\right)\widetilde{T}(q)\right\|_{p}<\varepsilon^{2},\]
for all $T\in \Hom_{\Gamma}(S,F,m,\delta,\sigma_{i}).$ Thus the proceeding lemma implies that if
\[e_{i}=\chi_{(\varepsilon,\infty)}\left(\left|\sum_{s\in \Gamma}f(s)\sigma_{i}(s)-1\right|\right),\]
then for all large $i$ we have
\[\|T(q)-e_{i}T(q)\|_{p}<\varepsilon,\]
\[\tr(e_{i})\leq \tr(\sigma_{i}(q))+2^{p}\kappa^{p}.\]
We identify $\alpha_{S}$ as a map into $L^{p}(M_{d_{i}}(\CC),\tr).$ Then
\[\alpha_{S}(\Hom_{\Gamma}(S,F,m,\delta,\sigma_{i}))\subseteq_{\varepsilon}\{e_{i}A:A\in L^{p}(M_{d_{i}}(\CC),\tr)\}.\]
So
\[\frac{1}{d_{i}}d_{\varepsilon}(\alpha_{S}(\Hom_{\Gamma}(S,F,m,\delta,\sigma_{i}),\|\cdot\|_{p})\leq\frac{1}{d_{i}}\Tr(e_{i})= \tr(e_{i})\leq \tr(\sigma_{i}(q))+2^{p}\kappa^{p}\]
and
\[\tr(\sigma_{i}(q))\to \tau(q)\]
as $i\to \infty.$
Taking the limit supremum over $(F,m,\delta)$ and then letting $\varepsilon\to 0$ proves that
\[\dim_{\Sigma,S^{p},\textnormal{mult}}(L^{p}(M,\tau),\Gamma)\leq \tau(q)+2^{p}\kappa^{p}.\]
Since $\kappa>0$ is arbitrary, this proves the claim.

\end{proof}

\begin{lemma}\label{L:volpack} Fix $1\leq p\leq \infty$ and a sequence of positive integers $d(n)\to \infty$. Let $\mu_{n}$ be the Lebesgue measure on $L^{p}(M_{d_(n)}(\CC),\frac{1}{d(n)}\Tr)$ normalized so that 
\[\mu_{n}(\Ball(L^{p}(M_{d(n)}(\CC),\frac{1}{d(n)}\Tr)))=1.\]
 Further, let $q_{n}\in \Proj(M_{d(n)}(\CC))$ be such that $\frac{1}{d(n)}\Tr(q_{n})$ converges to a positive real number. Then there is a function
\[\kappa\colon (0,1)\times (0,\infty)\to [0,1]\]
such that
\[\lim_{\varepsilon\to 0}\kappa(\alpha,\varepsilon)=1,\mbox{for all $\alpha>0$},\]
which satisfies the following property. For all measurable $A_{n}\subseteq \Ball(L^{p}(M_{d(n)}(\CC),\frac{1}{d(n)}\Tr))$ and $\alpha>0$ with
\[\limsup_{n\to \infty}\mu_{n}(A_{n})^{1/2d(n)^{2}}\geq \alpha,\]
we have for all $\varepsilon>0,$
\[\limsup_{n\to \infty}\frac{1}{d(n)\Tr(q_{n})}d_{\varepsilon}(A_{n}q_{n},\|\cdot\|_{p})\geq \kappa(\alpha,\varepsilon).\]

\end{lemma}

\begin{proof} Fix $1>\varepsilon>0$ and $\alpha>0.$ Suppose that $A_{n}$ is a sequence of meausrable subsets of $\Ball(L^{p}(M_{d(n)}(\CC),\frac{1}{d(n)}\Tr)$ with
\[\limsup_{n\to\infty}\mu_{n}(A_{n})^{1/2d(n)^{2}}\geq \alpha.\]
Suppose that $\kappa>0$ is such that
\[\limsup_{n\to \infty}\frac{1}{d(n)\Tr(q_{n})}d_{\varepsilon}(A_{n}q_{n},\|\cdot\|_{p})<\kappa.\]
We wish to get a lower bound on $\kappa.$ For all large $n$
\[d_{\varepsilon}(A_{n}q_{n},\|\cdot\|_{V_{n}})<d(n)\kappa\Tr(q_{n}) .\]
Let $W_{n}$ be a subspace of dimension at most $d(n)\kappa \Tr(q_{n})$ which $\varepsilon$-contains $A_{n}q_{n}.$ Thus
\[A_{n}q_{n}\subseteq (1+\varepsilon)\Ball(W_{n})+\varespilon \Ball\left(L^{p}(M_{d(n)}(\CC),\tr)q_{n}\right).\]
Let $S\subseteq (1+\varepsilon)\Ball(W_{n})$ be a maximal family of $\varepsilon$-separated vectors, i.e. for all $x,y\in S$ with $x\ne y$ we have $\|x-y\|\geq \varepsilon.$ Then the $\varepsilon/3$ balls centered at points in $S$ are disjoint and so by a volume computation
\[|S|\leq \left(\frac{3+3\varepsilon}{\varepsilon}\right)^{2\dim(W_{n})}.\]
By maximality, $S$ is $\varepsilon$-dense in $(1+\varepsilon)\Ball(W_{n}).$ Thus
\[A_{n}q_{n}\subseteq\bigcup_{x\in S}x+2\varepsilon\Ball\left(L^{p}(M_{d(n)}(\CC),\tr)q_{n}\right),\]
so
\[\vol(A_{n}q_{n})\leq 2^{2d(n)\tr(q_{n})}\varepsilon^{2d(n)^{2}\tr(q_{n})-2\dim(W_{n})}\left(3+3\varepsilon\right)^{2\dim(W_{n})}a_{p}(q_{n}),\]
where for $q\in \Proj(M_{d(n)}(\CC))$ we use
\[a_{p}(q)=\vol(\Ball\left(L^{p}(M_{d(n)}(\CC),\tr)q\right)).\]
Since $A_{n}\subseteq A_{n}q_{n}\times \Ball\left(L^{p}(M_{d(n)}(\CC),\tr\right),$ we have
\[\vol(A_{n})\leq \vol(A_{n}q_{n})a_{p}(1-q_{n})\leq \varepsilon^{2d(n)\Tr(q_{n})-2\dim(W_{n})}\left(3+3\varepsilon\right)^{2\dim(W_{n})}a_{p}(q_{n})a_{p}(1-q_{n}).\]
Thus
\begin{align*}
\alpha&\leq \limsup_{n\to\infty}\mu_{n}(A_{n})^{1/2d(n)^{2}}\\
&\leq \limsup_{n\to\infty}\varepsilon^{\tr(q_{n})-\tr(q_{n})\frac{\dim(W_{n})}{\Tr(q_{n})d(n)}}(3+3\varepsilon)^{\tr(q_{n})\frac{\dim(W_{n})}{\Tr(q_{n})d(n)}}\left(\frac{a_{p}(q_{n})a_{p}(1-q_{n})}{a_{p}(\id_{d(n)})}\right)^{1/2d(n)^{2}}.
\end{align*}
Let $q=\lim_{n\to\infty}\tr(q_{n}).$ Since
\[\limsup_{n\to \infty}\frac{1}{d(n)\Tr(q_{n})}\dim(W_{n})<\kappa\]
and $0<\varepsilon<1,$ we find that
\[\alpha\leq 6\cdot \varepsilon^{q(1-\kappa)}\limsup_{n\to\infty}\left(\frac{a_{p}(q_{n})a_{p}(1-q_{n})}{a_{p}(\id_{d(n)})}\right)^{1/2d(n)^{2}}.\]
	Hence it suffices to show that
\begin{equation}\label{E:TBPMFO}
\limsup_{n\to \infty}\left(\frac{a_{p}(q_{n})a_{p}(1-q_{n})}{a_{p}(\id_{d(n)})}\right)^{1/2d(n)^{2}}<\infty.
\end{equation}
It is well known that
\[a_{2}(q)=\frac{\pi^{\Tr(q)}}{\Tr(q)!}d(n)^{-d(n)}.\]
Since $\frac{1}{d(n)}\Tr(q_{n})$ converges to a positive real number, we may apply Stirling's formula and the above equation to see that there is a $M>1$ with
\[M^{-1}\leq \left(\frac{a_{2}(q_{n})a_{2}(1-q_{n})}{a_{2}(\id_{d(n)})}\right)^{1/2d(n)^{2}}<M.\]
We know by \cite{TJ} that there is a constant $C>0$ with
\begin{align*}
\left(\frac{a_{p}(q_{n})a_{p}(1-q_{n})}{a_{p}(\id_{d(n)})}\right)^{1/2d(n)^{2}}&\leq C\left(\frac{a_{p}(q_{n})a_{p}(1-q_{n})}{a_{2}(\id_{d(n)})}\right)^{1/2d(n)^{2}}\\
&\leq CM^{2}\left(\frac{a_{p}(q_{n})}{a_{2}(q_{n})}\right)^{1/2d(n)^{2}}\left(\frac{a_{p}(1-q_{n})}{a_{2}(1-q_{n})}\right)^{1/2d(n)^{2}}.
\end{align*}
	
	Let $p'$ be such that $\frac{1}{p}+\frac{1}{p'}=1.$ By the Santalo inequality (see \cite{Pis} Corollary 7.2), and the fact that $\frac{1}{d(n)}\Tr(q_{n})$ converges to a positive real number, we may find an $A>0$ with
\[\left(\frac{a_{p}(q_{n})}{a_{2}(q_{n})}\right)^{1/2d(n)^{2}}\left(\frac{a_{p}(1-q_{n})}{a_{2}(1-q_{n})}\right)^{1/2d(n)^{2}}\leq A\left(\frac{a_{2}(q_{n})}{a_{p'}(q_{n})}\right)^{1/2d(n)^{2}}\left(\frac{a_{2}(1-q_{n})}{a_{p'}(1-q_{n})}\right)^{1/2d(n)^{2}}\]
\[\leq AM^{2}\left(\frac{a_{2}(\id)}{a_{p'}(q_{n})a_{p'}(1-q_{n})}\right)^{1/2d(n)^{2}}.\]
Again by \cite{TJ}, we can find some $D>0$ with
\[\left(\frac{a_{2}(\id)}{a_{p'}(q_{n})a_{p'}(1-q_{n})}\right)^{1/2d(n)^{2}}\leq D\left(\frac{a_{p'}(\id)}{a_{p'}(q_{n})a_{p'}(1-q_{n})}\right)^{1/2d(n)^{2}}.\]
As
\[\Ball(L^{p'}(M_{d_{i}}(\CC),\tr))\subseteq \Ball(L^{p'}(M_{d_{i}}(\CC),\tr)q_{n})\times \Ball(L^{p'}(M_{d_{i}}(\CC),\tr)(1-q_{n})),\]
we find that
\[\left(\frac{a_{p'}(\id)}{a_{p'}(q_{n})a_{p'}(1-q_{n})}\right)^{1/2d(n)^{2}}\leq 1.\]
Putting all these inequalities together, we find that
\[\left(\frac{a_{p}(q_{n})a_{p}(1-q_{n})}{a_{p}(\id_{d(n)})}\right)^{1/2d(n)^{2}}\leq ACM^{4}D,\]
and this proves (\ref{E:TBPMFO}).

\end{proof}

To complete the calculation, it suffices to prove the following Theorem.

\begin{theorem}\label{T:LPcalc}
  Let $\Gamma$ be an $\R^{\omega}$-embeddable group and $\Sigma$ an embedding sequence. Let $M=L(\Gamma)$ and $\tau$ the canonical group trace on $M.$ Then for all $1\leq p<\infty$ and for every $q_{1},\dots,q_{n}\in \Proj(M)$ we have
\[\dim_{\Sigma,S^{p},\mult}\left(\bigoplus_{j=1}^{n}L^{p}(M,\tau)q_{j},\Gamma\right)=\underline{\dim}_{\Sigma,S^{p},\mult}\left(\bigoplus_{j=1}^{n}L^{p}(M,\tau)q_{j},\Gamma\right)\]
\[=\sum_{j=1}^{n}\tau(q_{j}).\]
\end{theorem}

\begin{proof} We use the generating sequence $S=(q_{1},\dots,q_{n},0,\dots)$ to do the calculation. By Proposition $\ref{P:LPupbound},$ we have the upper bound. So it suffices to prove the lower bound. By Lemma 5.5 in \cite{Me}, we can find maps (not assumed to be linear) $\rho_{i}\colon L(\Gamma)\to M_{d_{i}}(\CC)$ such that
\[\rho_{i}(\lambda(g))=\sigma_{i}(g),\mbox{ for $g\in \Gamma$}\]
\[\sup_{i}\|\rho_{i}(x)\|_{\infty}<\infty,\mbox{ for all $x\in L(\Gamma),$}\]
\[\tr(\rho_{i}(x))\to \tau(x),\mbox{ for all $x\in L(\Gamma),$}\]
\[\|P(\rho_{i}(x_{1}),\dots,\rho_{i}(x_{n})-\rho_{i}(P(x_{1},\dots,x_{n}))\|_{2}\to 0,\]
for all $x_{1},\dots,x_{n}\in L(\Gamma),$ and all $*$-polynomials $P$ in $n$-noncommuting variables. As in Proposition \ref{P:LPupbound}, we may assume that $\rho_{i}(q_{j})$ is an orthogonal projection for all $i,j.$

	Fix $F\subseteq \Gamma$ finite $m\geq n$ in $\NN,$ $\delta>0.$ Let $E\subseteq \Gamma$ be a finite set which is sufficiently large in a manner to be determined later. Let
\[V_{E}^{(j)}=\Span\{u_{g}q:g\in E\}.\]
For $A\in M_{d_{i}}(\CC)$ and  $E\subseteq \Gamma$ finite define
\[T_{A}^{(j)}\left(\sum_{g\in E}c_{g}u_{g}q\right)=\sum_{g\in E}c_{g}\sigma_{i}(g)\rho_{i}(q_{j})A.\]
Note that
\[\left\|T_{A}^{(j)}\left(\sum_{g\in E}c_{g}u_{g}q\right)\right\|_{p}\leq \|A\|_{\infty}\left\|\sum_{g\in E}c_{g}\sigma_{i}(g)\rho_{i}(q_{j})\right\|_{p}.\]
Since $\sigma_{i}$ is an embedding sequence, we know that
\[\left\|\sum_{g\in E}c_{g}\sigma_{i}(g)\rho_{i}(q_{j})\right\|_{p}\to \left\|\sum_{g\in E}c_{g}u_{q}q_{j}\right\|_{p}\]
pointwise. As $V_{E}^{(j)}$ is finite-dimensional,
\[\left\|\sum_{g\in E}c_{g}\sigma_{i}(g)\rho_{i}(q_{j})\right\|_{p}\to \left\|\sum_{g\in E}c_{g}u_{q}q_{j}\right\|_{p}\]
uniformly on the $\|\cdot\|_{p}$ unit ball of $V_{E}^{(j)}.$

	If $E$ is sufficiently large, then for all $g_{1},\dots,g_{k}\in F$
\begin{align*}
\|T_{A}^{(j)}(g_{1}\cdots g_{k}q_{j})-\sigma_{1}(g_{1})\cdots \sigma_{i}(g_{k})T_{A}^{(j)}(q_{j})\|_{p}&=\|\sigma_{i}(g_{1}\cdots g_{k})\rho_{i}(q_{j})A-\sigma_{1}(g_{1})\cdots \sigma_{i}(g_{k})\rho_{i}(q_{j})A\|_{p}\\
&\leq\|A\|_{\infty}\|\sigma_{i}(g_{1}\cdots g_{k})-\sigma_{i}(g_{1})\cdots \sigma_{i}(g_{k})\|_{p}\\
&\to 0.
\end{align*}
Thus if $E$ is sufficiently large, depending upon $F,m,\delta$ then for all $A_{1},\dots,A_{n}\in M_{d_{i}}(\CC)$ with $\|A_{j}\|_{\infty}\leq 1,$
\[T_{A_{1}}^{(1)}\oplus \cdots \oplus T_{A_{n}}^{(n)}\in \Hom_{\Gamma}(S,F,m\delta,\sigma_{i})_{n}.\]
So
\[\alpha_{S}( \Hom_{\Gamma}(S,F,m\delta,\sigma_{i})_{n})\supseteq \prod_{j=1}^{n}\Ball(M_{d_{i}}(\CC),\|\cdot\|_{\infty})\rho_{i}(q_{j}).\]
By \cite{TJ}
\[\inf_{i}\left(\frac{\vol(\Ball(M_{d_{i}}(\CC),\|\cdot\|_{\infty}))}{\vol\left(\Ball\left(M_{d_{i}}(\CC),\|\cdot\|_{L^{p}\left(M_{d_{i}}(\CC),\frac{1}{d_{i}}\Tr\right)}\right)\right)}\right)^{1/2d_{i}^{2}}>0,\]
so the theorem now follows from Lemma \ref{L:volpack}.

\end{proof}

We can prove an analogue for the action of $\Gamma$ on its reduced $C^{*}$-algebra but first we need a Lemma.

\begin{lemma} Let $\Gamma$ be a countable discrete group, and $V\subseteq L^{p}(L(\Gamma),\tau_{\Gamma})$ a closed $\Gamma$-invariant subspace (for the action of left multiplication by elements of $\Gamma).$ Then there is an orthogonal projection $q\in L(\Gamma)$ with $V=L^{p}(L(\Gamma),\tau_{\Gamma})q.$

\end{lemma}

\begin{proof} We always have the inequality
\[\|xy\|_{p}\leq \|x\|_{\infty}\|y\|_{p}.\]
Note that if $x_{n}\in L(\Gamma),\sup_{n}\|x_{n}\|_{\infty}<\infty,$ and $x_{n}\to x$ in the strong operator topology on $l^{2}(\Gamma),$ then $x_{n}y\to xy.$ Indeed, this follows by the above inequality and the density of $l^{2}(\Gamma)$ in $L^{p}(L(\Gamma),\tau_{\Gamma}).$ Thus a closed $\Gamma$-invariant subspace is the same as an $L(\Gamma)$-invariant subspace. It suffices to prove the following two claims.

	\emph{Claim 1. If $x\in L^{p}(L(\Gamma),\tau_{\Gamma}),$ then $\overline{L(\Gamma)x}^{\|\cdot\|_{p}}=L^{p}(L(\Gamma),\tau_{\Gamma})\chi_{(0,\infty)}(|x|)$}.

\emph{Claim 2. If $e,f$ are orthogonal projections in $L(\Gamma),$ then}
\[\overline{L^{p}(L(\Gamma),\tau_{\Gamma})e+L^{p}(L(\Gamma),\tau_{\Gamma})f}=L^{p}(L(\Gamma),\tau_{\Gamma})(e\vee f).\]
Indeed, if we grant the two claims, then by separability we can find increasing subspaces $V_{n}$ of $\Gamma$ of the form $L^{p}(L(\Gamma),\tau_{\Gamma})q_{n}$ for some orthogonal projection $q_{n}$ such that
\[V=\overline{\bigcup_{n=1}^{\infty}V_{n}}.\]
 Setting $q=\sup q_{n}$ we see that
\[V=L^{p}(L(\Gamma),\tau_{\Gamma})q.\]

	For Claim 2 it suffices to note that by functional calculus
\[1-(e\vee f)=1-(1-e)\wedge (1-f)=1-\lim_{n\to \infty}((1-e)(1-f)(1-e))^{n},\]
the limit in $\|\cdot\|_{p}.$ As
\[1-[(1-e)(1-f)(1-e)]^{n}\in L^{p}(L(\Gamma),\tau_{\Gamma})e+L^{p}(L(\Gamma),\tau_{\Gamma})f\]
for all $n,$ this implies that
\[L^{p}(L(\Gamma),\tau_{\Gamma})(e\vee f)\subseteq \overline{L^{p}(L(\Gamma),\tau_{\Gamma})e+L^{p}(L(\Gamma),\tau_{\Gamma})f}.\]
The reverse inclusion being trivial, this proves claim 2.

For Claim 1, let $x=v|x|$ be the polar decomposition. Since $|x|=v^{*}x,$
\[\overline{L(\Gamma) x}^{\|\cdot\|_{p}}=\overline{L(\Gamma)|x|}^{\|\cdot\|_{p}}.\]
Let
\[y_{n}=\chi_{(\varepsilon,\infty)}(|x|)|x|^{-1}.\]
By functional calculus
\[\|y_{n}|x|-\chi_{(0,\infty)}(|x|)\|_{p}\to 0.\]
Thus
\[\overline{L(\Gamma)|x|}^{\|\cdot\|_{p}}\supseteq L^{p}(\Gamma,\tau_{\Gamma})\chi_{(0,\infty)}(|x|).\]
The reverse inclusion being trivial, we are done.

\end{proof}

 If $\Gamma$ is a countable discrete group we use $C^{*}_{\lambda}(\Gamma)$ for $\overline{\CC[\Gamma]}^{\|\cdot\|_{\infty}}$ with the closure taken in the left regular representation. As a corollary of the above Theorem we deduce one of the conjectures stated in \cite{Me}.
\begin{cor} Let $\Gamma$ be an $\R^{\omega}$-embeddable group and $1\leq p<\infty.$ Let $I\subseteq C^{*}_{\lambda}(\Gamma)$ be a norm closed left-ideal. Let $\overline{I}^{\wk^{*}}=L(\Gamma)q$ (with the closure taken in $L(\Gamma).$ Then
\[\dim_{\Sigma,S^{p},\mult}(I,\Gamma)\geq \tau(q).\]
\end{cor}

\begin{proof} It suffices to show that the inclusion $I\subseteq L^{p}(L(\Gamma),\tau)q$ has dense image. By the previous Lemma we may find $q'\in \Proj(L(\Gamma))$ such that
\[\overline{I}^{\|\cdot\|_{p}}=L^{p}(L(\Gamma),\tau)q'.\]
By the argument in the previous Lemma,
\[q'=\sup_{x\in I}\chi_{(0,\infty)}(|x|).\]
So it suffices to prove the following two claims.

\emph{Claim 1. If $x\in C^{*}_{\lambda}(\Gamma),$ then $\chi_{(0,\infty)}(|x|)\in \overline{I}^{\wk^{*}}.$}

\emph{Claim 2. If $e,f\in \Proj(\overline{I}^{\wk^{*}}),$ then $e\vee f\in \Proj(\overline{I}^{\wk^{*}}).$}

	For the proof of Claim $1,$ let $x=v|x|$ be the polar decomposition. By the Kaplansky Density Theorem, we can find $v_{n}\in C^{*}_{\lambda}(\Gamma)$ with $\|v_{n}\|_{\infty}\leq 1$ and $\|v_{n}-v\|_{2}\to 0.$ But then $\|v_{n}^{*}x-|x|\|_{2}\to 0,$ so $|x|\in \overline{I}^{\wk^{*}}.$ Since
\[\chi_{(\varepsilon,\infty)}(|x|)=|x|^{-1}\chi_{(\varepsilon,\infty)}(|x|)|x|,\]
we find that $\chi_{(0,\infty)}(|x|)\in \overline{I}^{\wk^{*}}.$

	For the proof of claim 2, we use the formula (proved by functional calculus):
\[e\vee f=1-\lim_{n\to \infty}([(1-e)(1-f)(1-e)])^{n}\]
where the limit is in $\|\cdot\|_{2}.$ Since $e,f\in \overline{I}^{\wk^{*}},$ a little calculation shows that
\[1- ([(1-e)(1-f)(1-e)])^{n}\in \overline{I}^{\wk^{*}}.\]
This proves the corollary.

\end{proof}

We can also handle the case $p=\infty$ if we assume a little more.

\begin{definition}\emph{Let $A$ be a $C^{*}$-algebra. A sequence of potentially nonlinear, nonmultiplicative maps}
\[\sigma_{i}\colon A\to M_{d_{i}}(\CC)\]
 \emph{where $d_{i}$ is a sequence of integers going to $\infty$ are said to be} norm microstates \emph{if for all $a,b\in A$}
 \[\|\sigma_{i}(ab)-\sigma_{i}(a)\sigma_{i}(b)\|_{\infty}\to_{i\to\infty}0,\]
 \[\|\sigma_{i}(a)\|_{\infty}\to \|a\|,\]
 \[\|\sigma_{i}(a^{*})-\sigma_{i}(a)^{*}\|_{\infty}\to 0,\]
 \[\|\sigma_{i}(a+b)-\sigma_{i}(a)-\sigma_{i}(b)\|_{\infty}\to 0,\]
 \[\|\sigma_{i}(\lambda a)-\lambda\sigma_{i}(a)\|_{\infty}\to 0,\mbox{ for all $\lambda\in\CC,a\in A$}.\]
 \end{definition}

\begin{theorem}\label{T:inftycase} Let $\Gamma$ be a countable discrete group. Assume that there are norm microstates $\sigma_{i}\colon C^{*}_{\lambda}(\Gamma)\to M_{d_{i}}(\CC)$ such that
\[\tr( \sigma_{i}(x))\to \tau(x)\mbox{ for all $x\in C^{*}_{\lambda}(\Gamma)$,}\]
\[\sigma_{i}(g)\in \mathcal{U}(d_{i})\mbox{ for all $g\in \Gamma$.}\]
 Let $I\subseteq C^{*}_{\lambda}(\Gamma)$ be a norm-closed left ideal and let $I^{\wk^{*}}=L(\Gamma)q,$ with $q\in \Proj(L(\Gamma)).$ Then,
\[\underline{\dim}_{\Sigma,S^{\infty},\textnormal{mult}}(I,\Gamma)\geq \tau(q).\]

\end{theorem}

\begin{proof} Let
\[A=\frac{l^{\infty}(\NN,(M_{d_{i}}(\CC))_{i=1}^{\infty})}{c_{0}(\NN,(M_{d_{i}}(\CC))_{i=1}^{\infty})}.\]
Our hypothesis implies that there is an isometric $*$-homomorphism
\[\sigma\colon C^{*}_{\lambda}(\Gamma) \to A,\]
such that
\[\sigma(u_{g})=\pi(\sigma_{1}(g),\sigma_{2}(g),\dots)\]
where
\[\pi\colon l^{\infty}(\NN,(M_{d_{i}}(\CC))_{i=1}^{\infty})\to A\]
is the quotient map. By using a Hamel basis for $C^{*}_{\lambda}(\Gamma),$ we may choose a sequence of linear maps
\[\phi_{i}\colon C^{*}_{\lambda}(\Gamma)\to M_{d_{i}}(\CC)\]
with
\[\sigma(x)=\pi(\phi_{1}(x),\phi_{2}(x),\dots).\]
As before, we may extend $\phi_{i}$ to an embedding sequence
\[\psi_{i}\colon L(\Gamma)\to M_{d_{i}}(\CC).\]

	Now let $\varepsilon>0,$ and choose a finite subset $E\subseteq \Gamma,l\in \NN,$ and $c_{gj}\in \CC,$ for $(g,j)\in E\times \{1,\dots,l\}$ with
\[\left\|q-\sum_{\substack{g\in E,\\1\leq j\leq l}}c_{gj}u_{g}x_{j}\right\|_{2}<\varepsilon.\]
Fix $E\subseteq F\subseteq \Gamma$ finite, $l\leq m\in \NN,\delta>0.$ Since all injective $*$-homomorphisms defined on $C^{*}$-algebras are isometric, it is easy to see that if we define $\rho_{i}=\frac{\phi_{i}\big|_{I_{F,m}}}{\left\|\phi_{i}\big|_{I_{F,m}}\right\|},$ then
\[\left\|\rho_{i}-\phi_{i}\big|_{I_{F,m}}\right\|\to 0.\]

	For $B\in M_{d_{i}}(\CC)$ define
\[T_{B}\colon I_{F,m}\to M_{d_{i}}(\CC),\]
by
\[T_{B}(x)=\rho_{i}(x)B.\]
If $\|B\|_{\infty}\leq 1,$ then for all $x\in I_{F,m}$
\[\|T_{B}(x)\|\leq \|B\|_{\infty}\|\rho_{i}(x)\|_{\infty}\leq \|B\|_{\infty}\|x\|_{\infty}.\]
Further if $\|B\|_{\infty}\leq 1,$ if $1\leq j,k\leq m$ and $g_{1},\dots,g_{k}\in F,$ then
\begin{align*}
\|T_{B}(g_{1}\cdots g_{k}x_{j})-\sigma_{i}(g_{1})\cdots\sigma_{i}(g_{k})T_{B}(x_{j})\|&\leq \|\phi_{i}(g_{1}\cdots g_{k}x_{j})-\sigma_{i}(g_{1})\cdots\sigma_{i}(g_{k})\phi_{i}(x_{j})\|\\
&\to 0
\end{align*}
using that
\[\pi((\phi_{i}(g_{1}\cdots g_{k}x_{j}))_{i=1}^{\infty})=\pi((\sigma_{i}(g_{1})\cdots\sigma_{i}(g_{k})\phi_{i}(x_{j}))_{i=1}^{\infty})).\]

	Now suppose $V\subseteq l^{\infty}(\NN,M_{d_{i}}(\CC))$ $\varepsilon$-contains $\{(\rho_{i}(x_{j})B)_{j=1}^{\infty}:\|B\|_{\infty}\leq 1\}.$ Define a map $\Phi\colon l^{\infty}(\NN,M_{d_{i}}(\CC))\to L^{2}(M_{d_{i}}(\CC),\tr)$ by
\[\Phi(f)=\sum_{\substack{g\in E,1\leq j\leq l}}c_{gj}\sigma_{i}(g)f(j).\]
Our hypotheses imply that for all large $i,$
\[\Phi(V)\supseteq_{3\varepsilon,\|\cdot\|_{2}}\{qB:B\in \Ball(M_{d_{i}}(\CC),\|\cdot\|_{\infty})\}.\]
Our methods to prove Theorem \ref{T:LPcalc} can be used to complete the proof.

\end{proof}

Let us note that the assumption that $\sigma_{i}(g)\in \mathcal{U}(d_{i})$ is mild. For example, given any sequence of norm microstates
\[\sigma_{i}\colon \Gamma\to C^{*}_{\lambda}(\Gamma)\]
there always exists $U_{i,g}\in \mathcal{U}(d_{i})$ such that
\[\|\sigma_{i}(g)-U_{i,g}\|\to_{i\to\infty}0.\]
To see this, observe that
\[\|\sigma_{i}(g)^{*}\sigma_{i}(g)-1\|\to_{i\to\infty}0,\]
\[\|\sigma_{i}(g)\sigma_{i}(g)^{*}-1\|\to_{i\to\infty}0.\]
This implies that for all large $i,$ both $\sigma_{i}(g)^{*}\sigma_{i}(g),\sigma_{i}(g)\sigma_{i}(g)^{*}$ are invertible. Thus for all large $i,$ $\sigma_{i}(g)$ is both left and right invertible, and thus invertible. Hence if we let $\sigma_{i}(g)=U_{i,g}|\sigma_{i}(g)|$ be the polar decomposition, then for all large $i,$ we know $U_{i,g}$ is a unitary. Further,
\[\|\sigma_{i}(g)-U_{i,g}\|\leq \| |\sigma_{i}(g)|-1\|.\]
Since $|\sigma_{i}(g)|=(\sigma_{i}(g)^{*}\sigma_{i}(g))^{1/2},$ and
\[\|\sigma_{i}(g)^{*}\sigma_{i}(g)-1\|_{\infty}\to 0,\]
uniform continuity of the square root function and continuous functional calculus imply that
\[\||\sigma_{i}(g)|-1\|_{\infty}\to 0.\]
Thus
\[\|\sigma_{i}(g)-U_{i,g}\|_{\infty}\to 0.\]
We close this section by giving a Proposition which provides examples of groups which have a norm microstates as in the previous theorem.  We first state a Lemma that which will ease the process of proving that a group has such norm microstates.
	
	We need some preliminary notation. Let $A$ be a set and $(V_{a})_{a\in A}$ be Banach spaces. We let $l^{\infty}(A,(V_{a})_{a\in A})$ be the set of $(v_{a})_{a\in A}$ such that $v_{a}\in V_{a}$ and $\sup_{a}\|v_{a}\|<\infty.$ We let $c_{0}(A,(V_{a})_{a\in A})$ be the set of all $(v_{a})_{a\in A}$ such that $v_{a}\in V_{a}$ and $a\mapsto \|v_{a}\|$ is in $c_{0}(A).$ We need to consider the special case when $V_{a}$ are matrix algebras. Given $d(a)\in \NN,$ and a free ultrafilter $\omega\in\beta A\setminus A$ we let
\[\tau_{\omega}\colon\frac{l^{\infty}(A,(M_{d(a)}(\CC))_{a\in A})}{c_{0}(A,(M_{d(a)}(\CC))_{a\in A})}\to \CC\]
be defined by
\[\tau_{\omega}((A_{a})_{a\in A}+c_{0}(A,(M_{d(a)}(\CC))_{a\in A}))=\lim_{a\to\omega}\tr(A_{a}).\]

\begin{lemma} Let $\Gamma$ be a countable discrete group, and suppose that  $d_{i}$ is a sequence of integers with
\[d_{i}\to \infty.\]
Let $\sigma_{i}\colon C^{*}_{\lambda}(\Gamma)\to M_{d_{i}}(\CC)$ be a sequence of potentially nonlinear, nonmultiplicative maps.
Then $\sigma_{i}$ is a sequence of norm microstates such that
\[\tr(\sigma_{i}(x))\to \tau(x)\mbox{ for all $x\in C^{*}_{\lambda}(\Gamma)$}\]
if and only if there is an injective $*$-homomorphism
\[\sigma\colon C^{*}_{\lambda}(\Gamma)\to\frac{l^{\infty}(\NN,(M_{d_{i}}(\CC))_{i=1}^{\infty})}{c_{0}(\NN,(M_{d_{i}}(\CC))_{i=1}^{\infty})}\]
satisfying
\[\sigma(x)=(\sigma_{i}(x))_{i=1}^{\infty}+c_{0}(\NN,(M_{d_{i}}(\CC))_{i=1}^{\infty}),\]
\[\tau_{\omega}\circ \sigma=\tau,\]
for all $\omega\in\beta\NN\setminus \NN.$

\end{lemma}
\begin{proof}
First suppose that there is an injective $*$-homomorphism
\[\sigma\colon C^{*}_{\lambda}(\Gamma)\to\frac{l^{\infty}(\NN,(M_{d_{i}}(\CC))_{i=1}^{\infty})}{c_{0}(\NN,(M_{d_{i}}(\CC))_{i=1}^{\infty})}\]
with
\[\sigma(x)=(\sigma_{i}(x))_{i=1}^{\infty}+c_{0}(\NN,(M_{d_{i}}(\CC))_{i=1}^{\infty}),\]
\[\tau_{\omega}\circ \sigma=\tau.\]
By definition of the norm on $\frac{l^{\infty}(\NN,(M_{d_{i}}(\CC))_{i=1}^{\infty})}{c_{0}(\NN,(M_{d_{i}}(\CC))_{i=1}^{\infty})}$ we know that
\[\|x\|_{\infty}=\limsup_{i\to\infty}\|\sigma_{i}(x)\|_{\infty}.\]
The statement that $\tau_{\omega}\circ\sigma=\tau$ is equivalent to
\[\lim_{i\to\omega}\tr(\sigma_{i}(x))=\tau(x).\]
Since we are assuming this for all free ultrafilters $\omega$ we have
\[\lim_{i\to\infty}\tr(\sigma_{i}(x))=\tau(x),\]
for all $x\in C^{*}_{\lambda}(\Gamma).$ As noted in Proposition 3.1 of \cite{Male}, this implies that
\[\liminf_{i\to\infty}\|\sigma_{i}(x)\|_{\infty}\geq \|x\|_{\infty},\]
for all $x\in C^{*}_{\lambda}(\Gamma).$ We thus see that $\sigma_{i}$ is a sequence of norm microstates such that
\[\tr(\sigma_{i}(x))\to \tau(x)\]
for all $x\in C^{*}_{\lambda}(\Gamma).$ The converse is even easier.

\end{proof}
For part $(i)$ of the next Proposition we recall that a discrete group $\Gamma$ is \emph{maximally almost periodic} if it has an injective homomorphism into a compact topological group.

\begin{proposition}\label{P:examples} Let $\mathcal{C}$ be the class of discrete groups $\Gamma$ which have norm microstates $\sigma_{i}\colon C^{*}_{\lambda}(\Gamma)\to M_{d_{i}}(\CC)$ such that
\[\tr( \sigma_{i}(x))\to \tau(x)\]
for all $x\in C^{*}_{\lambda}(\Gamma).$

(i): $\Gamma\in \mathcal{C}$ if and only if, for every finite $F\subseteq \QQ[i](\Gamma),$ for every $\varepsilon>0,$ there is a map
\[\sigma\colon \QQ[i](\Gamma)\to M_{d}(\CC)\]
for some $d\in \NN$ with
\[\|\sigma(ab)-\sigma(a)\sigma(b)\|<\varepsilon,\mbox{ for all $a,b\in F$,}\]
\[|\|\sigma(a)\|-\|a\||<\varepsilon,\mbox{ for all $a\in F$},\]
\[|\tr(\sigma(a))-\tau(a)|<\varepsilon,\mbox{ for all $a\in F$}.\]

(ii): Every maximally almost periodic amenable group is in $\mathcal{C}.$

(iii): If $\Gamma_{1},\Gamma_{2}\in\mathcal{C},$ and $\Gamma_{1}$ is exact (see \cite{BO} Definition 5.1.1), then $\Gamma_{1}\times \Gamma_{2}\in \mathcal{C}.$

(iv): Suppose $\Gamma$ is a discrete group and $(\Gamma_{n})_{n\in\NN}$ is an increasing sequence of subgroups of $\Gamma$ with
\[\Gamma=\bigcup_{n=1}^{\infty}\Gamma_{n}.\]
If $\Gamma_{n}\in \mathcal{C}$ for all $n\in\NN,$ then $\Gamma\in \mathcal{C}.$

(v): If $\Gamma_{1},\Gamma_{2}\in\mathcal{C},$ then $\Gamma_{1}*\Gamma_{2}\in\mathcal{C}.$

\end{proposition}

\begin{proof}

For all five parts we will use the preceding Lemma.

(i): From the hypothesis and a diagonal argument, we may find a sequence of integers $d_{k}$ and maps
\[\sigma_{k}\colon \QQ[i](\Gamma)\to M_{d_{k}}(\CC)\]
such that for all $x,y\in \QQ[i](\Gamma),\lambda\in\QQ[i]$
\[\|\sigma_{k}(x)\|\to_{k\to\infty}\|x\|,\]
\[\|\sigma_{k}(xy)-\sigma_{k}(x)\sigma_{k}(y)\|\to_{k\to\infty}0,\]
\[\tr(\sigma_{k}(x))\to \tau(x),\]
\[\|\sigma_{k}(x+y)-\sigma_{k}(x)-\sigma_{k}(y)\|\to_{k\to\infty}0,\]
\[\|\sigma_{k}(\lambda x)-\lambda\sigma_{k}(x)\|\to_{k\to\infty}0.\]
We thus have an isometric $*$-homomorphism
\[\sigma\colon \QQ[i](\Gamma)\to \frac{l^{\infty}(\NN,(M_{d_{k}}(\CC))_{k=1}^{\infty})}{c_{0}(\NN,(M_{d_{k}}(\CC))_{k=1}^{\infty})},\]
with
\[\sigma(x)=(\sigma_{k}(x))+c_{0}(\NN,(M_{d_{i}}(\CC))_{k=1}^{\infty})\]
for all $x\in \QQ[i](\Gamma).$ We may extend by $\sigma$ continuity to $C^{*}_{\lambda}(\Gamma)$ to an isometric map
\[\sigma\colon C^{*}_{\lambda}(\Gamma)\to \frac{l^{\infty}(\NN,(M_{d_{k}}(\CC))_{k=1}^{\infty})}{c_{0}(\NN,(M_{d_{k}}(\CC))_{k=1}^{\infty})}.\]
It is not hard to check that $\tau_{\omega}\circ \sigma=\tau$ for every $\omega\in\beta\NN\setminus\NN.$

(ii): This is a consequence of Corollary 4.15 of \cite{BrownMeans}.

(iii): For this part, to avoid ambiguity, we let $\tau_{j}$ be the trace on $C^{*}_{\lambda}(\Gamma_{j}).$ We shall then use $\tau$ for the trace on $C^{*}_{\lambda}(\Gamma_{1}\times\Gamma_{2}).$ We use the notation as in the preceding Lemma. Also, if $A,B$ are $C^{*}$-algebras, we use $A\otimes_{\textnormal{min}}B$ for the minimal tensor product of $A$ and $B$ (see \cite{BO} Definition 3.3.4). We shall use that
\[C^{*}_{\lambda}(\Gamma_{1}\times \Gamma_{2})=C^{*}_{\lambda}(\Gamma_{1})\otimes_{\textnormal{min}} C^{*}_{\lambda}(\Gamma_{2}).\]
 Let
\[\sigma^{(j)}_{k}\colon C^{*}_{\lambda}(\Gamma_{j})\to M_{d_{k}^{(j)}}(\CC),j=1,2\]
be a sequence of norm microstates with
\[\tr(\sigma^{(j)}_{k}(x))\to \tau(x)\mbox{ for all $x\in C^{*}_{\lambda}(\Gamma_{j})$}.\]
For $j=1,2$ we let
\[\sigma^{(j)}\colon C^{*}_{\lambda}(\Gamma_{j})\to\frac{l^{\infty}(\NN,(M_{d_{k}^{(j)}}(\CC))_{k=1}^{\infty})}{c_{0}(\NN,(M_{d_{k}^{(j)}}(\CC))_{k=1}^{\infty})}\]
be the injective $*$-homomorphism defined by
\[\sigma^{(j)}(x)=(\sigma_{k}^{(j)}(x))_{k=1}^{\infty}+c_{0}(\NN,(M_{d_{k}^{(j)}}(\CC))_{k=1}^{\infty}).\]
Since $\Gamma_{1}$ is exact, we have inclusions
\begin{align*}
C^{*}_{\lambda}(\Gamma_{1})\otimes_{\textnormal{min}} C^{*}_{\lambda}(\Gamma_{2})&\to C^{*}_{\lambda}(\Gamma_{1})\otimes_{\textnormal{min}}\frac{l^{\infty}(\NN,(M_{d_{k}^{(2)}}(\CC))_{k=1}^{\infty})}{c_{0}(\NN,(M_{d_{k}^{(2)}}(\CC))_{i=1}^{\infty})}\\
&\cong\frac{C^{*}_{\lambda}(\Gamma_{1})\otimes_{\textnormal{min}}l^{\infty}(\NN,(M_{d_{k}^{(2)}}(\CC))_{k=1}^{\infty})}{C^{*}_{\lambda}(\Gamma_{1})\otimes_{\textnormal{min}}c_{0}(\NN,(M_{d_{k}^{(2)}}(\CC))_{k=1}^{\infty})}\\
&\to \frac{l^{\infty}(\NN,(C^{*}_{\lambda}(\Gamma_{1})\otimes_{\textnormal{min}}M_{d_{k}^{(2)}}(\CC))_{k=1}^{\infty})}{c_{0}(\NN,(C^{*}_{\lambda}(\Gamma_{1})\otimes_{\textnormal{min}}M_{d_{k}^{(2)}}(\CC))_{k=1}^{\infty})},
\end{align*}
the first line being given by $\id\otimes \sigma^{(2)},$ and the second line following from exactness of $C^{*}_{\lambda}(\Gamma_{1}).$ We let
\[\phi\colon C^{*}_{\lambda}(\Gamma_{1})\otimes_{\textnormal{min}} C^{*}_{\lambda}(\Gamma_{2})\to \frac{l^{\infty}(\NN,(C^{*}_{\lambda}(\Gamma_{1})\otimes_{\textnormal{min}}M_{d_{k}^{(2)}}(\CC))_{k=1}^{\infty})}{c_{0}(\NN,(C^{*}_{\lambda}(\Gamma_{1})\otimes_{\textnormal{min}}M_{d_{k}^{(2)}}(\CC))_{k=1}^{\infty})}\]
be the composition of these inclusion maps. If $\omega\in\beta\NN\setminus\NN$ is a free ultrafilter, we let
\[\tau_{1}\otimes\tau_{\omega}\colon  \frac{l^{\infty}(\NN,(C^{*}_{\lambda}(\Gamma_{1})\otimes_{\textnormal{min}}M_{d_{k}^{(2)}}(\CC))_{k=1}^{\infty})}{c_{0}(\NN,(C^{*}_{\lambda}(\Gamma_{1})\otimes_{\textnormal{min}}M_{d_{k}^{(2)}}(\CC))_{k=1}^{\infty})}\to \CC\]
be defined by
\[\tau_{1}\otimes\tau_{\omega}((x_{k})_{k=1}^{\infty}+c_{0}(\NN,(C^{*}_{\lambda}(\Gamma_{1})\otimes_{\textnormal{min}}M_{d_{k}^{(2)}}(\CC))_{k=1}^{\infty}))=\lim_{k\to\omega}\tau_{1}\otimes\tr(x_{k}).\]
It is easy to see that
\[\tau_{1}\otimes\tau_{\omega}\circ \phi=\tau\]
for all $\omega\in\beta\NN\setminus\NN.$ As in the preceding Lemma, we have that for all $x_{1},$ $\dots$, $x_{p}\in C^{*}_{\lambda}(\Gamma_{1})$, $y_{1},\dots,y_{p}\in C^{*}_{\lambda}(\Gamma_{2}),$
\[\left\|\sum_{j=1}^{p}x_{j}\otimes y_{j}\right\|=\lim_{k\to\infty}\left\|\sum_{j=1}^{p}x_{j}\otimes \sigma_{k}^{(2)}(y_{j})\right\|.\]
But now applying the same argument (using exactness of $M_{d_{k}^{(2)}}(\CC)$) we have for all $k\in\NN,$
\[\left\|\sum_{j=1}^{p}x_{j}\otimes \sigma_{k}^{(2)}(y_{j})\right\|=\lim_{l\to\infty}\left\|\sum_{j=1}^{p}\sigma_{l}^{(1)}(x_{j})\otimes \sigma_{k}^{(2)}(y_{j})\right\|.\]
So
\[\left\|\sum_{j=1}^{k}x_{j}\otimes y_{j}\right\|=\lim_{k\to\infty}\lim_{l\to\infty}\left\|\sum_{j=1}^{p}\sigma_{l}^{(1)}(x_{j})\otimes \sigma_{k}^{(2)}(y_{j})\right\|.\]
Additionally, it is easy to see that
\[\lim_{k\to\infty}\lim_{l\to\infty}\sum_{j=1}^{p}\tr(\sigma_{l}^{(1)}(x_{j}))\tr(\sigma_{k}^{(2)}(x_{j}))=\sum_{j=1}^{p}\tau_{1}(x_{j})\tau_{2}(y_{j}).\]
From these two limiting statements, it is not hard to argue that the conditions of $(i)$ hold.

(iv): Since every finite subset of $\QQ[i](\Gamma)$ is contained in $\QQ[i](\Gamma_{n})$ for some $n\in\NN,$ and we have an isometric, trace-preserving inclusion
\[C^{*}_{\lambda}(\Gamma_{n})\subseteq C^{*}_{\lambda}(\Gamma)\]
this is obvious from $(i).$

(v): As noted in $(iv),$ if $\Lambda$ is a subgroup of $\Gamma,$ then we have a canonical trace-preserving isometric inclusion
\[C^{*}_{\lambda}(\Lambda)\subseteq C^{*}_{\lambda}(\Gamma).\]
From this observation and $(iv)$ it suffices to assume that $\Gamma_{1},\Gamma_{2}$ are finitely generated. Let $s_{1}^{(j)},\dots,s_{t_{j}}^{(j)}$ be generators of $\Gamma_{j},j=1,2.$ Let
\[\pi \colon\FF_{t_{1}+t_{2}}\to C^{*}_{\lambda}(\Gamma_{1}*\Gamma_{2})\]
be the unique homomorphism defined by
\[\pi(a_{l})=\begin{cases}
s_{l}^{(1)},&\textnormal{ if $1\leq l\leq t_{1}$}\\
s_{l-t_{1}}^{(2)},& \textnormal{ if $t_{1}< l\leq t_{1}+t_{2}$.}
\end{cases}\]
We denote by $\pi$ its unique linear extension to a map
\[\CC(\FF_{t_{1}+t_{2}})\to \CC(\Gamma_{1}*\Gamma_{2}).\]
Let $d_{k}^{(j)},j=1,2$ be a sequence of integers and
\[\sigma^{(j)}_{k}\colon C^{*}_{\lambda}(\Gamma_{j})\to M_{d_{k}^{(j)}}(\CC)\]
be a sequence of norm microstates with
\[\tr\circ \sigma^{(j)}_{k}(x)\to_{k\to\infty} \tau(x)\]
for all $x\in C^{*}_{\lambda}(\Gamma).$  By replacing $\sigma^{(1)}_{k},\sigma^{(2)}_{k}$ with the maps
\[C^{*}_{\lambda}(\Gamma_{1})\to M_{d_{k}^{(1)}}(\CC)\otimes M_{d_{k}^{(2)}}(\CC),x\mapsto \sigma_{k}^{(1)}(x)\otimes 1,\]
\[C^{*}_{\lambda}(\Gamma_{2})\to M_{d_{k}^{(1)}}(\CC)\otimes M_{d_{k}^{(2)}}(\CC),x\mapsto 1\otimes \sigma_{k}^{(2)}(x),\]
we may assume that $d_{k}^{(1)}=d_{k}^{(2)}.$ Hence, we shall use $d_{k}$ for $d_{k}^{(1)}$ (or $d_{k}^{(2)}$). By the remarks after Theorem \ref{T:inftycase}  , we may assume that $\sigma^{(j)}_{k}(g)\in \mathcal{U}(d_{k})$ for every $g\in\Gamma_{j}.$ For any $U\in\mathcal{U}(d_{k}),$ let
\[\pi_{k,U}\colon \FF_{t_{1}+t_{2}}\to \mathcal{U}(d_{k})\]
be the unique homomorphism defined by
\[\pi(a_{l})=\begin{cases}
U^{*}\sigma_{k}^{(1)}(s_{l}^{(1)})U,&\textnormal{ if $1\leq l\leq t_{1}$}\\
\sigma_{k}^{(2)}(s_{l-t_{1}}^{(2)}),& \textnormal{ if $t_{1}< l\leq t_{1}+t_{2}$.}
\end{cases}\]
By \cite{MaleCollins} Theorem 1.5, there is a sequence $U_{k}\in \mathcal{U}(d_{k})$ such that
\[\|\pi_{k,U_{k}}(x)\|\to \|\pi(x)\|\mbox{ for all $x\in \CC(\FF_{t_{1}+t_{2}})$},\]
\[\tr(\pi_{k,U_{k}}(x))\to \tau(x)\mbox{ for al $x\in \CC(\FF_{t_{1}+t_{2}})$}.\]
From this it is not hard to show that the conditions of $(i)$ hold.

\end{proof}

We note that by \cite{BrownMeans} Proposition 4.1.4 and Choi-Effros lifting theorem (see \cite{BO} Theorem C.3), if every amenable group were in $\mathcal{C}$ it would follow that $C^{*}_{\lambda}(\Gamma)$ is quasidiagonal for every amenable group. This is known as Rosenberg's conjecture, and is an open problem of major current interest (see e.g. \cite{OzRoSa},\cite{JoMaEc}). The case $\Gamma_{1}=\Gamma_{2}=\ZZ$ of $(v)$ was first proved by Haagerup and Thorbj\o rnsen in \cite{HTExt} Theorem B (combining with \cite{VoicRandom}  Theorem 3.8).

\section{ Definition of $l^{p}$-Dimension Using Vectors}

	In this section, we give a definition of the extended von Neumann dimension using vectors instead of almost equivariant operators. This may be conceptually simpler, as we do not have to deal with the technicalities involving changing domains inherent to the definition of $\Hom_{\Gamma}(\cdots).$ The definition is  much simpler and requires fewer preliminaries as well. However, for many theoretical purposes it will still be easier to use the notion of almost equivariant operators. We will give this alternate definition after the following lemma.

\begin{lemma} Let $V$ be a finite-dimensional Banach space, let $B$ be a finite set and $(v_{\beta})_{\beta\in B}\in V^{B}$ such that $V=\Span\{v_{\beta};\beta\in B\}.$ Then for any $\eta>0,$ there is a $\delta>0$ with the following property. If $Y$ is a Banach space and $(\xi_{\beta})_{\beta\in B}\in Y^{B}$ have the property that for all $c\in l^{1}(B)$ with $\|c\|_{1}\leq 1$
\[\left\|\sum_{\beta\in B}c(\beta)\xi_{\beta}\right\|\leq \delta+\left\|\sum_{\beta\in B}c(\beta)v_{\beta}\right\|,\]
then there is a $T\colon V\to Y$ with $\|T\|\leq 1$ such that
\[\|T(v_{\beta})-\xi_{\beta}\|<\eta,\]
for all $\beta\in B.$
\end{lemma}

\begin{proof} Let $A\subseteq B$ be such that $\{v_{\alpha}:\alpha\in A\}$ is a basis for $V.$ For $Y$ and $(\xi_{\beta})_{\beta\in B}$ as in the statement of the Lemma, let $\widetidle{T}\colon V\to Y$ be the unique linear operator satisfying
\[\widetidle{T}(v_{\alpha})=\xi_{\alpha},\]
for $\alpha\in A.$ By finite-dimensionality, there is a $C_{V}>0$  with
\[\sum_{\alpha\in A}|c_{\alpha}|\leq C_{V}\left\|\sum_{\alpha\in A}c_{\alpha}v_{\alpha}\right\| .\]
Fix $\xi\in V,$ with $\|\xi\|=1,$ and write
\[\xi=\sum_{\alpha\in A}c_{\alpha}v_{\alpha}.\]
Then by hypothesis,
\[\frac{\|\widetidle{T}(\xi)\|}{\sum_{\alpha\in A}|c_{\alpha}|}\leq \delta+\frac{\left\|\sum_{\alpha\in A}c_{\alpha}v_{\alpha}\right\|}{\sum_{\alpha\in A}|c_{\alpha}|}=\delta+\frac{1}{\sum_{\alpha\in A}|c_{\alpha}|}.\]
So
\[\|\widetilde{T}(\xi)\|\leq 1+\delta\sum_{\alpha\in A}|c_{\alpha}|\leq 1+\delta C_{V}.\]
As $\xi$ was arbitrary,
\[\|\widetidle{T}\|\leq C_{V}\delta+1.\]

	Set $T=\frac{1}{1+C_{V}\delta}\widetidle{T},$ then $\|T\|\leq 1.$ For each $\beta\in B\setminus A$ choose $a_{\alpha}^{(\beta)},\alpha\in A$ such that
\[v_{\beta}=\sum_{\alpha\in A}a_{\alpha}^{(\beta)}v_{\alpha}.\]
For $\beta\in B\setminus A,$ let
\[A_{\beta}=\sum_{\alpha\in A}|a_{\alpha}^{(\beta)}|.\]
For $\beta\in B\setminus A,$ define $c^{(\beta)}\in l^{1}(B)$ by
\[c^{(\beta)}(\alpha)=\frac{a_{\alpha}^{(\beta)}}{1+A_{\beta}},\mbox{ for all $\alpha\in A$}\]
\[c^{(\beta)}(\beta)=-\frac{1}{1+A_{\beta}},\]
\[c^{(\beta)}(\beta')=0,\mbox{ for all $\beta'\in B\setminus(A\cup \{\beta\})$}.\]
Then for $\beta\in B\setminus A,$ $\|c^{(\beta)}\|_{1}=1$ and
\[\sum_{\beta'\in B}c^{(\beta)}(\beta')v_{\beta}=0.\]
Thus by our hypothesis for $\beta\in B\setminus A,$
\[\frac{1}{1+A_{\beta}}\|\xi_{\alpha}-\widetidle{T}(v_{\alpha})\|=\left\|\sum_{\beta'\in B}c^{(\beta)}(\beta')\xi_{\beta'}\right\|\leq \delta,\]
so
\[\|\xi_{\beta}-\widetilde{T}(\xi_{\beta})\|\leq (1+A_{\beta})\delta.\]
Trivially for all $\alpha\in A$ we have
\[\|\xi_{\alpha}-\widetilde{T}(\xi_{\alpha})\|=0.\]
For all $\beta\in B,$
\[\|\widetilde{T}(v_{\beta})-T(v_{\beta})\|=\left|1-\frac{1}{1+\delta C_{V}}\right|\|\widetilde{T}(v_{\beta})\|\leq \delta C_{V}\|v_{\beta}\|.\]
Set
\[M=\max\left(\max_{\beta\in B\setminus A}1+A_{\beta}+C_{V}\|v_{\beta}\|,\max_{\alpha\in A}C_{V}\|v_{\alpha}\|\right).\]
Then $M$ does not depend upon $Y,$ or $\varepsilon$ and for all $\beta\in B$
\[\|T(v_{\beta})-\xi_{\beta}\|\leq 2M\delta,\]
so if $\delta<\frac{\eta}{2M},$ we are done.

\end{proof}

\begin{definition} \emph{Let $V$ be a Banach space with a uniformly bounded action of a countable discrete group $\Gamma$ and $\sigma_{i}\colon \Gamma\to \Isom(V_{i})$ with $V_{i}$ finite-dimensional. We let $\Vect_{\Gamma}(S,F,m,\delta,\sigma_{i})$ be all $m$-tuples $(\xi_{j})_{j=1}^{m}$ of  vectors in $V$ such that for all $(c_{g_{1},\dots,g_{l},j})_{1\leq l,j\leq m,g_{1},\dots,g_{l}\in F}$ with $\sum_{\substack{g_{1},\dots,g_{l}\in F\\1\leq j,l\leq m}}|c_{g_{1},\dots,g_{m},j}|\leq 1,$ we have}
\[\left\|\sum_{\substack{g_{1},\dots,g_{l}\in F\\1\leq j.l\leq m}}c_{g_{1},\dots,g_{l},j}\sigma_{i}(g_{1})\cdots\sigma_{i}(g_{m})\xi_{j}\right\|\leq \delta+\left\|\sum_{\substack{g_{1},\dots,g_{l}\in F\\1\leq j,l\leq m}}c_{g_{1},\dots,g_{l},j}g_{1}\cdots g_{l}x_{j}\right\|.\]
\emph{Set}
\[\vdim_{\Sigma}(S,F,m,\delta,\varepsilon,\rho)=\limsup_{i\to \infty}\frac{1}{\dim V_{i}}d_{\varepsilon}(\Vect_{\Gamma}(S,F,m,\delta,\sigma_{i}),\rho_{V_{i}}),\]
\[\vdim_{\Sigma}(S,\varepsilon,\rho)=\inf_{F,m,\delta}\vdim_{\Sigma}(S,F,m,\delta,\varepsilon,\rho),\]
\[\vdim_{\Sigma}(S,\rho)=\sup_{\varepsilon>0}\vdim_{\Sigma}(S,\varepsilon,\rho).\]

\end{definition}

\begin{proposition}Let $V$ be a Banach space with a uniformly bounded action of a countable discrete group $\Gamma$ and $\sigma_{i}\colon \Gamma\to \Isom(V_{i})$ with $V_{i}$ finite-dimensional. Then for any dynamically generating sequence $S,$ and any product norm $\rho,$
\[\dim_{\Sigma}(V,\Gamma)=\vdim_{\Sigma}(S,\rho).\]
\end{proposition}

\begin{proof} Let $S=(x_{j})_{j=1}^{\infty}.$  Fix $F\subseteq \Gamma$ finite with $e\in F$, $m\in \NN,$ $\delta>0.$ Suppose that $T\in \Hom_{\Gamma}(S,F,m,\delta,\sigma_{i})$ and set $\xi_{j}=T(x_{j}).$ Then for all $(c_{g_{1},\dots,g_{l},j})_{g_{1},\dots,g_{l}\in F,1\leq j,l\leq m}$ with
\[\sum_{\substack{g_{1},\dots,g_{l}\in F\\1\leq j,l\leq m}}|c_{g_{1},\dots,g_{l},j}|\leq 1,\]
we have
\begin{align*}
\left\|\sum_{\substack{g_{1},\dots,g_{l}\in F\\1\leq j,l\leq m}}c_{g_{1},\dots,g_{m},j}\sigma_{i}(g_{1})\cdots\sigma_{i}(g_{l})\xi_{j}\right\|&\leq\delta+\left\|T\left(\sum_{\substack{g_{1},\dots,g_{l}\in F\\1\leq j,l\leq m}}c_{g_{1},\dots,g_{l},j}g_{1}\cdots g_{l}\xi_{j}\right)\right\|\\
&\leq \delta+\left\|\sum_{\substack{g_{1},\dots,g_{l}\in F\\1\leq j,l\leq m}}c_{g_{1},\dots,g_{l},j}g_{1}\cdots g_{l}\xi_{j}\right\|.
\end{align*}
So $(\xi_{j})_{j=1}^{m}\in \Vect_{\Gamma}(S,F,m,\delta,\sigma_{i})$ and $\vdim\leq \dim.$

	For the opposite inequality, let $\varepsilon>0$ and let $M=\sup_{j}\|x_{j}\|.$ Since $\rho$ is a product norm, we may find an  $N\in \NN$ and a $\kappa>0$ such that if $f\in l^{\infty}(\NN),$ if $\|f\|_{\infty}\leq M$ and
\[\max_{1\leq j\leq N}|f(j)|<\kappa,\]
then
\[\rho(f)<\varepsilon.\]
Fix $F\subseteq \Gamma$ finite with $e\in F$  and $m\in \NN$ with $m\geq N.$ Let $\delta'>0$ be sufficiently small depending upon $\kappa,$ in a manner to be determined later. Set
\[B=\bigsqcup_{l=1}^{m}\{(g_{1},\dots,g_{l},j):g_{1},\dots,g_{l}\in F,1\leq j\leq m\},\]
\[V=V_{F,m},\]
\[v_{\beta}=g_{1}\cdots g_{l}x_{j},\mbox{ if $\beta=(g_{1},\dots,g_{l},j)\in B$},\]
\[\eta=\delta'.\]
	Let $\delta>0$ be as in the preceding lemma for this $B,V,(v_{\beta})_{\beta\in B},\eta.$  If $(\xi_{j})_{j=1}^{m}\in \Vect_{\Gamma}(S,F,m,\delta,\sigma_{i}),$ then  by the preceding lemma we can find a $T\colon V_{F,m}\to V_{i}$ with $\|T\|\leq 1$ and $\|T(g_{1}\cdots g_{l}x_{j})-\sigma_{i}(g_{1})\cdots \sigma_{i}(g_{l})\xi_{j}\|<\delta',$ for all $g_{1},\dots,g_{l}\in F,1\leq j,l\leq m.$  Thus for all $1\leq j,l\leq m,g_{1},\dots,g_{l}\in F,$
\[\|T(g_{1}\cdots g_{l}x_{j})-\sigma_{i}(g_{1})\cdots \sigma_{i}(g_{l})T(x_{j})\|<2\delta'.\]
Thus $T\in \Hom_{\Gamma}(S,F,m,2\delta',\sigma_{i}),$ and
\[\max_{1\leq j\leq m}\|T(x_{j})-\xi_{j}\|<\delta',\]
since $e\in F.$ So if we choose $\delta'<\kappa,$  then since $m\geq N,$ our choice of $\kappa$ implies
\[\alpha_{S}(\Hom_{\Gamma}(S,F,m,\delta,\sigma_{i}))\subseteq_{\varepsilon,\rho_{V_{i}}}\Vect_{\Gamma}(S,F,m,\delta,\sigma_{i}),\]
so
\[d_{2\varepsilon}(\alpha_{S}(\Hom_{\Gamma}(S,F,m,\delta,\sigma_{i}),\rho_{V_{i}})\leq d_{\varepsilon}(\Vect_{\Gamma}(S,F,m,\delta,\sigma_{i}),\rho_{V_{i}}).\]
Taking limits in the appropriate order, we see that $\dim\leq \vdim.$

\end{proof}

\section{$l^{p}$-Betti Numbers of Free Groups}\label{S:l^{p}Betti}

	Let $X$ be a CW complex and let $\Delta_{n}(X)$ be the collection of $n$-simplices of $X.$ Suppose that $\Gamma$ acts properly on $X$ with compact quotient, preserving the simplicial structure. For $v_{0},\dots,v_{n}\in X,$ let
\[[v_{0},v_{1},\dots,v_{n}]\]
be the simplex spanned by $v_{0},\dots,v_{n}.$ Let
\[V_{n}(X)=\{(v_{0},\dots,v_{n})\in X:[v_{0},\dots,v_{n}]\in \Delta_{n}(X)\}.\]
We abuse notation and let $l^{p}(\Delta_{n}(X))$ for $1\leq p\leq \infty$ be the set of all functions $f\colon V_{n}(X)\to \CC$ such that
\[f(v_{\sigma(0)},\dots,v_{\sigma(n)})=(\sgn \sigma)f(v_{0},\dots,v_{n}),\mbox{ for $\sigma\in \Sym(\{0,\dots,n\}),$}\]
\[\sum_{[v_{0},\dots,v_{n}]\in \Delta_{n}(X)}|f(v_{0},\dots,v_{n})|^{p}<\infty,\mbox{ for $p<\infty$,}\]
\[\sup_{[v_{0},\dots,v_{n}]\in \Delta_{n}(X)}|f( v_{0},\dots,v_{n})|<\infty,\mbox{ for$p=\infty$.}\]
  By our antisymmetry condition the above sum is unchanged if we use a different representative for $[v_{0},\dots,v_{n}].$ On $l^{p}(\Delta_{n}(X))$ we use the norm
\[\|f\|_{p}^{p}=\sum_{v\in \Delta_{n}(X)}|f(v_{0},\dots,v_{n})|^{p},\mbox{ for $p<\infty$,}\]
\[\|f\|_{\infty}=\sup_{[v_{0},\dots,v_{n}]\in \Delta_{n}(X)}|f(v_{0},\dots,v_{n})|.\]
Define the discrete differential $\delta\colon l^{p}(\Delta_{n-1}(X))\to l^{p}(\Delta_{n}(X))$ by
\[(\delta f)(v_{0},\dots,v_{n})=\sum_{j=0}^{n}(-1)^{j}f(v_{0},\dots,\widehat{v_{j}},\dots,v_{n}),\]
where the hat indicates a term omitted. Note that $\delta f$ satisfies the appropriate antisymmetry condition. Define the $n^{th}$ $l^{p}$-cohomology space of $X$ by
\[H^{n}_{l^{p}}(X)=\frac{\ker(\delta)\cap l^{p}(\Delta_{n}(X))}{\overline{\delta(l^{p}(\Delta_{n-1}(X))}}.\]
We define the $l^{p}$-Betti numbers of $X$ with respect to $\Gamma$ by
\[\beta^{(p)}_{\Sigma,n}(X,\Gamma)=\dim_{\Sigma,l^{p}}(H^{n}_{l^{p}}(X),\Gamma).\]

	The definition of $l^{p}$-cohomology is due to Gromov in \cite{GromovLP}. We also refer the reader to \cite{Pansu} for a survey on results on $l^{p}$-cohomology. It is known that if $X$ is contractible and $\pi_{1}(X/\Gamma)\cong \Gamma,$  then the $l^{p}$-cohomology space only depends upon $\Gamma$ (see \cite{GromovLP} page 219). If $\Gamma$ is sofic, we may use $l^{p}$-dimension to define
\[H^{n}_{l^{p}}(\Gamma)=H^{n}_{l^{p}}(X,\Gamma),\]
\[\beta^{(p)}_{\Sigma,n}(\Gamma)=\beta^{(p)}_{\Sigma,n}(X,\Gamma),\]
for such $X.$  We will call these the $l^{p}$-Betti numbers of $\Gamma$ (with respect to $\Sigma$). The definition above for $p=2$ goes back to Atiyah in \cite{Atiyah}. We remark that this is the first definition of an $l^{p}$-Betti numbers for a class of  groups. The reason for this is that one needs to use a dimenison function valid for actions on $l^{p}$-spaces in order to defined $l^{p}$-Betti numbers and this is what we have done in \cite{Me}.

	We also consider $l^{p}$-homology. Define $\del \colon l^{p}(\Delta_{n}(X))\to l^{p}(\Delta_{n-1}(X))$ by
\[\del f(v_{0},\dots, v_{n-1})=\sum_{x:[v_{0},\dots, v_{n-1},x]\in \Delta_{n}(X)}f(v_{0},\dots,v_{n-1},x).\]
We use $T^{t}$ for the Banach space adjoint of a bounded $T\colon V\to W$ between Banach spaces $V,W.$ By direct computation
\[(\del \colon l^{p'}(\Delta_{n}(X))\to l^{p'}(\Delta_{n-1}(X)))=(\delta\colon l^{p}(\Delta_{n-1}(X))\to l^{p}(\Delta_{n}(X)))^{t},\]
when $\frac{1}{p}+\frac{1}{p'}=1.$ Define the $l^{p}$-homology of $X$ by
\[H_{n}^{l^{p}}(X)=\frac{\ker(\del)\cap l^{p}(\Delta_{n}(X))}{\overline{\del(l^{p}(\Delta_{n+1}(X))}}.\]

	We shall be interested in the $l^{p}$-Betti numbers of free groups. Fix $n\in \NN$ and consider the free group $\FF_{n}$ on $n$ letters $a_{1},\dots,a_{n}.$ Let $G$ be the Cayley graph of $\FF_{n}$ with respect to $a_{1},\dots,a_{n},$ we regard the edges of $G$ as oriented. There is a natural $1$-dimensional CW complex $X$ associated to $G,$ whose $0$-simplices are the vertices of $G,$ and whose $1$-simplices are the edges of $G,$ and whose attaching maps are determined by incidence of edges in the natural way (see \cite{Ha} page 83).  Then $X$ is contractible, since $G$ is a tree. Also $\pi_{1}(X/\FF_{n})\cong \FF_{n},$ so the $l^{p}$-cohomology of $G$ is the $l^{p}$-cohomology of $\FF_{n}.$  Let $E(\FF_{n})$ denote the set of edges of $\FF_{n}.$ Then $l^{p}(E(\FF_{n}))$ as defined above is the set of all functions $f\colon E(\FF_{n})\to \CC$ such that
\[f(x,s)=-f(s,x)\mbox{ if $(s,x)\in E(\FF_{n})$},\]
\[\sum_{j=1}^{n}\sum_{x\in \FF_{n}}|f(x,xa_{j})|^{p}<\infty\]
with the norm
\[\|f\|_{p}^{p}=\sum_{j=1}^{n}\sum_{x\in \FF_{n}}|f(x,xa_{j})|^{p}.\]
Note that this is indeed a norm on $l^{p}(E(\FF_{n})),$ and that $\FF_{n}$ acts isometrically on $l^{p}(E(\FF_{n}))$ by left translation. Also $l^{p}(E(\FF_{n}))$ is isomorphic to $l^{p}(\FF_{n})^{\oplus n}$ with respect to this action. If $(x,s)\in E(\FF_{n}),$ we let $\E_{(x,s)}$ be the function on $E(\FF_{n})$ such that
\[\E_{(x,s)}(y,t)=0\mbox{ if $\{x,s\}\ne \{y,t\}$}\]
\[\E_{(x,s)}(x,s)=1\]
\[\E_{(x,s)}(s,x)=-1.\]
We think of $\E_{(x,s)}$ as representing the edge going from $x$ to $s.$

The discrete differential $\delta\colon l^{p}(\FF_{n})\to l^{p}(E(\FF_{n}))$ we defined above is given by
\[(\delta f)(x,s)=f(s)-f(x)\mbox{ $(x,s)\in E(\FF_{n}))$}.\]
The corresponding  $l^{p}$-cohomology space is given by
\[H^{1}_{l^{p}}(\FF_{n})=l^{p}(E(\FF_{n}))/\overline{\delta(l^{p}(\FF_{n})}.\]
Also, $\del\colon l^{p}(E(\FF_{n}))\to l^{p}(\FF_{n})$ is given by
\[(\del f)(x)=\sum_{j=1}^{n}f(x,xa_{j})-\sum_{j=1}^{n}f(xa_{j}^{-1},x).\]
In this section, we compute the $l^{p}$-Betti numbers
\[\beta_{\Sigma,1}^{(p)}(\FF_{n}),\]
for $1\leq p\leq 2.$

	Let $\Gamma$ be a countable discrete group. We define $\rho\colon \Gamma\to B(l^{p}(\Gamma))$ by
\[(\rho(g)f)(x)=f(xg).\]

\begin{lemma}\label{L:RIHTDT} Let $n\in \NN,$ with $n\geq 2.$ Fix $1\leq p<\infty.$ There is a $C>0$ such that
\[\|\delta f\|_{p}\geq C\|f\|_{p},\]
for all $f\in l^{p}(\FF_{n}).$ In particular, the image of $\delta$ is closed.
\end{lemma}

\begin{proof}
	Assume the lemma is false. Then we can find $f_{k}\in l^{p}(\FF_{n}),$ with $\|f_{k}\|_{p}=1$ and
\[\|\delta f_{k}\|_{p}\to 0.\]
By direct computation
\[\|\delta f_{k}\|_{p}^{p}=\sum_{j=1}^{n}\|\rho(a_{j})f_{k}-f_{k}\|_{p}^{p},\]
where $a_{1},\dots,a_{n}$ are the free generators of $\FF_{n}.$ By the triangle inequality,
\[\|\rho(a_{j})|f_{k}|-|f_{k}|\|_{p}\leq \|\rho(a_{j})f_{k}-f_{k}\|_{p}\]
so we may assume $f_{k}\geq 0.$ Let $g_{n}=f_{k}^{p},$ then $g_{k}\in l^{1}(\FF_{n})$ and $\|g_{k}\|_{1}=1.$
	
	By calculus, for real numbers $a,b\geq 0$ we have
\[|a^{p}-b^{p}|\leq p\max(|a|^{p-1},|b|^{p-1})|a-b|\leq p|a|^{p-1}|a-b|+p|b|^{p-1}|a-b|.\]
Thus
\begin{align*}
\|\rho(a_{j})g_{k}-g_{k}\|_{1}&\leq p\sum_{g\in \FF_{n}}|f_{k}(g)|^{p-1}|f_{k}(g)-f_{k}(ga_{j})|+|f_{k}(ga_{j})|^{p-1}||f_{k}(ga_{j})-f_{k}(g)|\\
&\leq 2p\|f_{k}-\rho(a_{j})f_{k}\|_{p},
\end{align*}
where in the last line we use H\"{o}lder's inequality and that $\|f_{k}\|_{p}=1.$  Thus
\[\|\rho(a_{j})g_{k}-g_{k}\|_{1}\to 0.\]
Since $\{a_{1},\dots,a_{n}\}$ generate $\FF_{n},$ it follows that
\[\|\rho(x)g_{k}-g_{k}\|_{1}\to 0\]
for all $x\in \FF_{n}.$ By \cite{BO} Theorem 2.6.8 (2), this implies that $\FF_{n}$ is amenable. It is well known that $\FF_{n}$ is not amenable so we have reached a contradiction.

\end{proof}

\begin{lemma}\label{L:generateco} Fix $n\in \NN,$ $1\leq p<\infty.$ Then the set of all images of the elements $\E_{(e,a_{1})},\dots,$ $\E_{(e,a_{n-1})}$ is dynamically generating for $H^{1}_{l^{p}}(\FF_{n}).$

\end{lemma}

\begin{proof}  It suffices to show that
\[W=\delta(l^{p}(\FF_{n}))+\Span\{\E_{(s,sa_{j})}:s\in \FF_{n},1\leq j\leq n-1\}\]
is norm dense in $l^{p}(E(\FF_{n})).$ It is enough to show that
\[\E_{(e,a_{n})}\in \overline{W}^{\|\cdot\|}.\]
By convexity it is enough to show that $\E_{(e,a_{n})}$ is in the weak closure of $W.$ We shall prove by induction on $k$ that
\[\E_{(e,a_{n})}\equiv \E_{(a_{n}^{k},a_{n}^{k+1})}\mod W.\]
This is enough since
\[\E_{(a_{n}^{k},a_{n}^{k+1})}\to 0\]
weakly.

	The base case $k=0$ is trivial, so assume the result true for some $k.$ Then
\begin{align*}
\E_{(a_{n}^{k},a_{n}^{k+1})}-\delta(\chi_{\{a_{n}^{k+1}\}})&=\sum_{j=1}^{n}\E_{(a_{n}^{k+1},a_{n}^{k+1}a_{j})}+\sum_{j=1}^{n-1}\E_{(a_{n}^{k+1},a_{n}^{k+1}a_{j}^{-1})}\\
&=\E_{(a_{n}^{k+1},a_{n}^{k+2})}+\sum_{j=1}^{n-1}a_{n}^{k+1}\E_{(e,a_{j})}-\sum_{j=1}^{n-1}a_{n}^{k+1}a_{j}^{-1}\E_{(e,a_{j})}\\
&\equiv\E_{(a_{n}^{k+1},a_{n}^{k+2})} \mod W.
\end{align*}
Here is a graphical explanation of the above calculation. If we think of the elements of $l^{p}(E(\FF_{n}))$ as formal sums of oriented edges, then $-\delta(\chi_{a_{n}^{k+1}})$ is a ``source" at $a_{n}^{k+1}.$ It is the sum of all edges adjacent to $a_{n}^{k+1},$ directed away from $a_{n}^{k+1}.$ Below is a graphical representation of $-\delta(\chi_{a_{n}^{k+1}}):$
\[\xymatrix{
&                           &  &a_{n}a_{n-1}^{-1} &a_{n}^{k+2} &a_{n}a_{n-1}\\
&-\delta(\chi_{a_{n}^{k+1}})&= & \vdots  &\ar[ul],\ar[dl]a_{n}^{k+1}\ar[u],\ar[d],\ar[ur],\ar[dr]&\vdots \\
&                           &  &a_{n}a_{1}^{-1}&a_{n}^{k} &a_{n}a_{1}}\]
The above computation can be phrased as follows:
\[-\delta(\chi_{a_{n}^{k+1}})+\mathcal{E}_{(a_{n}^{k},a_{n}^{k+1})}=\]
\[\xymatrix{
&a_{n}^{k+1}a_{n-1}    &a_{n}^{k+2}                                             & a_{n}^{k+1}a_{1}       &   &     \\
&\vdots                &\ar[ul],\ar[dl],\ar[d]a_{n}^{k+1} \ar[u],\ar[ur],\ar[dr] & \vdots               & + &a_{n}^{k+1}\\
&a_{n}^{k+1}a_{1}^{-1} &a_{n}^{k}                                            &a_{n}^{k+1}a_{n-1}^{-1}   &   &\ar[u]a_{n}^{k} \\
&                       &                                                         &=                    &   &}\]
\[\xymatrix{
&a_{n}^{k+2} & & a_{n}a_{n-1}^{-1}&                                      & a_{n}a_{n-1}\\
&a_{n}^{l+1}\ar[u] & +& \vdots          &\ar[ul],\ar[dl]a_{n}\ar[ur],\ar[dr] & \vdots\\
&            &  & a_{n}a_{1}^{-1}  &                                     & a_{n}a_{1}}\]
and the second term on the right-hand side is easily seen to be in the span of translates of $\mathcal{E}_{(e,a_{j})},j=1,\dots, n-1.$ This completes the induction step.

\end{proof}

We shall prove the analogous claim for $l^{p}$-homology of free groups, but we need a few preliminary results. These next few results must be well known, but we include proofs for completeness. For a countable discrete group $\Gamma,$ we let $\rho\colon \Gamma\to B(l^{p}(\Gamma)),1\leq p\leq \infty$ be given by
\[(\rho(g)f)(x)=f(xg),\mbox{ for $f\in l^{p}(\Gamma).$}\]

\begin{lemma}\label{L:strictcontract} Let $\Gamma$ be a non-amenable group with finite-generating set $S.$ Let $A\colon l^{p}(\Gamma)\to l^{p}(\Gamma)$ be defined by
\[A=\frac{1}{|S\cup S^{-1}|}\sum_{s\in S\cup S^{-1}}\rho(s).\]
For $1<p<\infty,$ there is a constant $C_{p}<1$ such that $\|Af\|_{p}\leq C_{p}\|f\|_{p}$ for all $f\in l^{p}(\Gamma).$

\end{lemma}

\begin{proof} We use
\[\|A\|_{l^{p}\to l^{p}}\]
for the norm of $A$ as an operator from $l^{p}(\Gamma)\to l^{p}(\Gamma).$ We know $\|A\|_{l^{2}\to l^{2}}<1$ from the non-amenability of $\Gamma$ (see \cite{BO} Theorem 2.6.8 (8)). Since $\|A\|_{l^{\infty}\to l^{\infty}}\leq 1$ and $\|A\|_{l^{1}\to l^{1}}\leq 1,$ it follows by interpolation that for all $1<p<\infty,$ there exists a $C_{p}<1$ with $\|A\|_{l^{p}\to l^{p}}\leq C_{p}.$ From this the lemma follows. \end{proof}

\begin{lemma} Let $n\in \NN$ with $n\geq 2.$ For $1<p<\infty,$ the operator $\del \circ \delta\colon l^{p}(\FF_{n})\to l^{p}(\FF_{n})$ is invertible.\end{lemma}
\begin{proof} Let $a_{1},\dots,a_{n}$ be free generators for $\FF_{n},$ and let $S=\{a_{1},\dots,a_{n}\}.$  We have that
\[\del(\delta f)(x)=\sum_{s\in S\cup S^{-1}}f(x)-f(xs)=|S\cup S^{-1}|\left(f(x)-\frac{1}{|S\cup S^{-1}|}\sum_{s\in S\cup S^{-1}}\rho(s)f(x)\right).\]
So
\begin{equation}\label{E:spellingthingsoutformorons}
\del\circ \delta=\left(\id-\frac{1}{|S\cup S^{-1}|}\sum_{s\in S\cup S^{-1}}\rho(s)\right)|S\cup S^{-1}|.
\end{equation}
By the previous lemma
\[\left\|\frac{1}{|S\cup S^{-1}|}\sum_{s\in S\cup S^{-1}}\rho(s)\right\|_{l^{p}\to l^{p}}<1,\]
for $1<p<\infty.$ By (\ref{E:spellingthingsoutformorons}) it follows that $\del(\delta)$ is invertible for $1<p<\infty.$
\end{proof}

For the next corollary we use the following notation: if $V,W,U$ are Banach spaces with $W,U\subseteq V,$ we use $V=W\oplus U$ to mean $W\cap U=\{0\},W+U=V.$

\begin{cor} Let $n\in \NN$ with $n\geq 2.$  For $1<p<\infty,$ we have the following Hodge Decomposition:
\[l^{p}(E(\FF_{n}))=\ker(\del \colon l^{p}(E(\FF_{n}))\to l^{p}(\Gamma))\oplus \delta(l^{p}(\FF_{n})).\]
 \end{cor}
\begin{proof} By  Lemma \ref{L:RIHTDT}, we know that $\delta(l^{p}(\FF_{n}))$ is closed in $l^{p}(E(\FF_{n})).$ It is clear that $\ker(\del \colon l^{p}(E(\FF_{n}))\to l^{p}(\Gamma))$ is closed in $l^{p}(E(\FF_{n})).$  Given $f\in \ker(\del \colon l^{p}(E(\FF_{n}))\to l^{p}(\FF_{n}))\cap \delta(l^{p}(\FF_{n}))$ write $f=\delta(g).$ Then
\[0=\del(f)=\del(\delta(g)).\]
By the preceding lemma we have that $g=0.$

	If $f\in l^{p}(E(\Gamma)),$ then by the preceding lemma we can find a unique $g$ with $\del(f)=\del(\delta(g)).$  Then $f-\delta(g)\in \ker(\del),$ and
\[f=f-\delta(g)+\delta(g).\]

\end{proof}

For the next proposition, recall that if $a_{1},\dots,a_{n}$ are free generators for $\FF_{n},$ then every $x\in \FF_{n}$ can be written as
\[x=a_{i_{1}}^{r_{1}}a_{i_{2}}^{r_{2}}\cdots a_{i_{k}}^{r_{k}},\]
where $i_{l}\ne i_{l+1}$ for $1\leq l\leq k,$ and $r_{k}\in \ZZ\setminus\{0\}.$ This will be call the \emph{reduced} expression of $x.$ We will say $x$ \emph{starts} with $a_{i_{1}}$ if $r_{1}>0,$ and that $x$ \emph{starts} with $a_{i_{1}}^{-1}$ if $r_{1}<0.$  We will call
\[\sum_{l=1}^{k}r_{l}\]
the \emph{word length} of $x.$ We use $|x|$ for the word length of $x.$

\begin{proposition} Let $n\in \NN,$  and $1<p<\infty.$ Then $H_{1}^{l^{p}}(\FF_{n})$ can be generated by $n-1$ elements.
\end{proposition}

\begin{proof} The claim for $n=1$ is clear since $H^{1}_{l^{p}}(\ZZ)=0.$ First, we show how to reduce to the case $n=2.$ Let $n>2,$ and let $a_{1},\dots,a_{n}$ be the generators of $\FF_{n}.$ Consider the injective homomorphisms $\phi_{j}\colon \FF_{2}\to \FF_{n}$ for $1\leq j\leq n-1$ given by $\phi_{j}(a_{i})=a_{i+j}.$ Let $f$ be an element in $l^{p}(E(\FF_{2}))$ such that $\Span(\FF_{2}f)$ is dense in $\ker(\del)\cap l^{p}(E(\FF_{2})).$ Let $f_{j}\in l^{p}(E(\FF_{n}))$ be the element defined by
\[f_{j}(x,y)=\begin{cases}
0,\textnormal{ if one of $x,y\notin \phi_{j}(\FF_{2})$}\\
f(\phi_{j}^{-1}(x),\phi_{j}^{-1}(y)),\textnormal{ otherwise.}\\
\end{cases}\]
Then $f_{j}\in \ker(\del).$ It is easy to see from the preceding corollary and the fact that $f$ generates $\ker(\del)\cap l^{p}(E(\FF_{2}))$ that
\[\mathcal{E}_{(e,a_{j})}\in \overline{\ker(\del)+\delta(l^{p}(\FF_{n}))}^{\|\cdot\|_{p}}.\]
Again by the preceding corollary we find that $f_{1},\dots,f_{n-1}$ generate $\ker(\del).$ Thus it suffices to handle the case $n=2.$

	We now concentrate on the case $n=2,$ and we use $a,b$ for the generators of $\FF_{2}.$ We define $f\colon E(\FF_{2})\to \RR$ as follows:
\[f(x,y)=\begin{cases}
\left(\frac{1}{3}\right)^{|x|},\textnormal{ $(x,y)\in\mathcal{E}(\FF_{n}),|y|=|x|+1$ $y$ starts with $a$ or $b$}\\
-\left(\frac{1}{3}\right)^{|y|},\textnormal{ $(x,y)\in \mathcal{E}(\FF_{n}),|x|=|y|+1$ $x$ starts with $a$ or $b$}\\
\left(\frac{1}{3}\right)^{|y|},\textnormal{ $(x,y)\in \mathcal{E}(\FF_{n}),|x|=|y|+1,$ $x$ starts with $a^{-1}$ or $b^{-1}$}\\
-\left(\frac{1}{3}\right)^{|x|}\textnormal{ $(x,y)\in \mathcal{E}(\FF_{n}), |y|=|x|+1,$ $y$ starts with $a^{-1}$ or $b^{-1}$}
\end{cases}.\]
The function $f$ is pictured below:
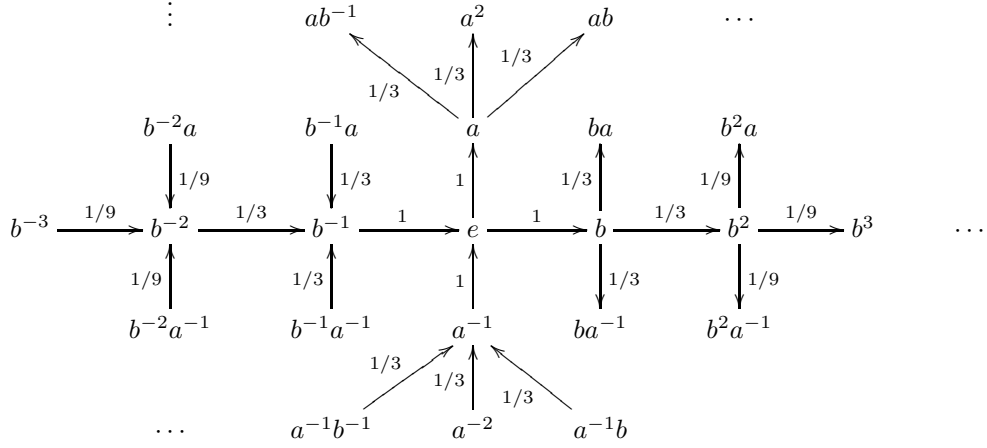
\[\xymatrix{
              &  \vdots  &ab^{-1}                       &             a^{2}                             &ab                         & \cdots  &  & &\\
               &b^{-2}a\ar[d]^{1/9}   &b^{-1}a \ar[d]^{1/3}          & \ar[ul]^{1/3} a \ar[u]^{1/3} \ar[ur]^{1/3}    &ba                         &b^{2}a   & & &\\
b^{-3}\ar[r]^{1/9}&b^{-2}\ar[r]^{1/3}&b^{-1}\ar[r]^{1}              &\ar[u]^{1}e\ar[r]^{1}                           &b \ar[d]^{1/3}\ar[u]^{1/3}\ar[r]^{1/3}& b^{2}\ar[u]^{1/9}\ar[r]^{1/9}\ar[d]^{1/9}&b^{3}& \cdots\\
& b^{-2}a^{-1}\ar[u]^{1/9} &b^{-1}a^{-1}\ar[u]^{1/3}     &\ar[u]^{1}a^{-1}                                &ba^{-1}        &b^{2}a^{-1}& & &\\
 &       \cdots            & a^{-1}b^{-1}\ar[ur]^{1/3}  & \ar[u]^{1/3} a^{-2}                            &\ar[ul]^{1/3} a^{-1}b      &&&& }\]
It is not hard to show that $\partial(f)=0.$ Since the number of words in the free group of length $n$ is $4\cdot 3^{n-1},$ we also see that $f\in l^{p}(\mathcal{E}(\FF_{n})).$  Set $V=\overline{\Span(\FF_{2}f)+\delta(l^{p}(\FF_{2}))}^{\wk}=\overline{\Span(\FF_{2}f)+\delta(l^{p}(\FF_{2}))}^{\|\cdot\|}.$
To show that $f$ generates $\ker(\del)$ it suffices by the preceding corollary to show that
\[\mathcal{E}_{(e,a_{1})},\mathcal{E}_{(e,a_{2})}\in V.\]
	
	Let $B_{n}=\{(x,y)\in G:|x|,|y|\leq n\}.$ For $n\geq 0,$ let $g_{n}\colon E(\FF_{n})\to \CC,$ be the function defined by
\begin{align*}
\chi_{B_{n}}g_{n}&=\left(\sum_{k=0}^{n-1}(1/3)^{n}\right)\left(\mathcal{E}_{(e,a)}+\mathcal{E}_{(e,b)}+\mathcal{E}_{(a^{-1},e)}+\mathcal{E}_{(b^{-1},e)}\right),\\
(1-\chi_{B_{n}})g_{n}&=(1-\chi_{B_{n}})f,
\end{align*}
we first show that $g_{n}\in \Span(\FF_{2}f)+\delta(l^{p}(\FF_{2}))$ for all $n\in \NN.$

	We prove this by induction on $n,$ the case $n=1$ being clear since $g_{1}=f.$ Suppose the claim true for some $n.$ Then, it is not hard to show that
\[g_{n+1}=g_{n}+\frac{1}{3^{n}}\sum_{\substack{ w\in \FF_{n}\setminus\{e\},\\ |w|\leq n, \\ w \textnormal{ starts with $a$ or $b$}}}\delta(\chi_{\{w\}})-\frac{1}{3^{n}}\sum_{\substack{ w\in \FF_{n}\setminus\{e\},\\ |w|\leq n, \\ w \textnormal{ starts with $a^{-1}$ or $b^{-1}$}}}\delta(\chi_{\{w\}}).\]
So inductively, we see that $g_{n}\in \Span(\FF_{2}f)+\delta(l^{p}(\FF_{2})),$  for all $n\in\NN.$ The first two steps of this process are pictured below:
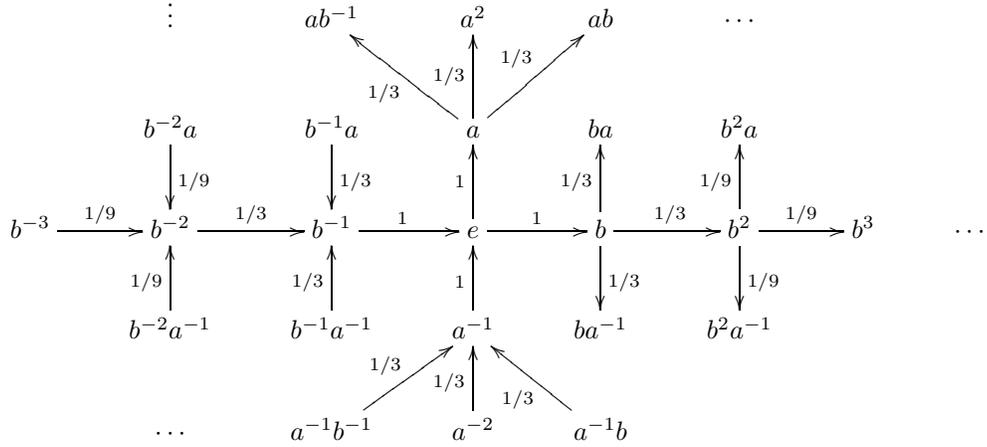
\[\xymatrix{
              &  \vdots  &ab^{-1}                       &             a^{2}                             &ab                         & \cdots  &  & &\\
               &b^{-2}a\ar[d]^{1/9}   &b^{-1}a \ar[d]^{1/3}          & \ar[ul]^{1/3} a \ar[u]^{1/3} \ar[ur]^{1/3}    &ba                         &b^{2}a   & & &\\
b^{-3}\ar[r]^{1/9}&b^{-2}\ar[r]^{1/3}&b^{-1}\ar[r]^{1}              &\ar[u]^{1}e\ar[r]^{1}                           &b \ar[u]^{1/3}\ar[r]^{1/3}\ar[d]^{1/3}& b^{2}\ar[u]^{1/9}\ar[r]^{1/9}\ar[d]^{1/9}&b^{3}& \cdots\\
& b^{-2}a^{-1}\ar[u]^{1/9} &b^{-1}a^{-1}\ar[u]^{1/3}     &\ar[u]^{1}a^{-1}                                &ba^{-1}       &b^{2}a^{-1}& & &\\
 &       \cdots            & a^{-1}b^{-1}\ar[ur]^{1/3}  & \ar[u]^{1/3} a^{-2}                            &\ar[ul]^{1/3} a^{-1}b      &&&& }\]
\[\xymatrix{&\ar @{=>}[r]^{\frac{1}{3}(\delta(\chi_{\{a\}})+\delta(\chi_{\{b\}})-\delta(\chi_{\{b^{-1}\}})-\delta(\chi_{\{a^{-1}\}}))} &}\]
\[\xymatrix{
              &  \vdots  &ab^{-1}                       &             a^{2}                             &ab                         & \cdots  &  & &\\
               &b^{-2}a\ar[d]^{1/9}   &b^{-1}a \ar[d]^{0}          & \ar[ul]^{0} a \ar[u]^{0} \ar[ur]^{0}    &ba                         &b^{2}a   & & &\\
b^{-3}\ar[r]^{1/9}&b^{-2}\ar[r]^{0}&b^{-1}\ar[r]^{4/3}              &\ar[u]^{4/3}e\ar[r]^{4/3}                           &b \ar[u]^{0}\ar[r]^{0}\ar[d]^{0}& b^{2}\ar[u]^{1/9}\ar[r]^{1/9}\ar[d]^{1/9}&b^{3}& \cdots\\
& b^{-2}a^{-1}\ar[u]^{1/9} &b^{-1}a^{-1}\ar[u]^{0}     &\ar[u]^{4/3}a^{-1}                                &ba^{-1}        &b^{2}a^{-1}& & &\\
 &       \cdots            & a^{-1}b^{-1}\ar[ur]^{0}  & \ar[u]^{0} a^{-2}                            &\ar[ul]^{0} a^{-1}b      &&&& }\]
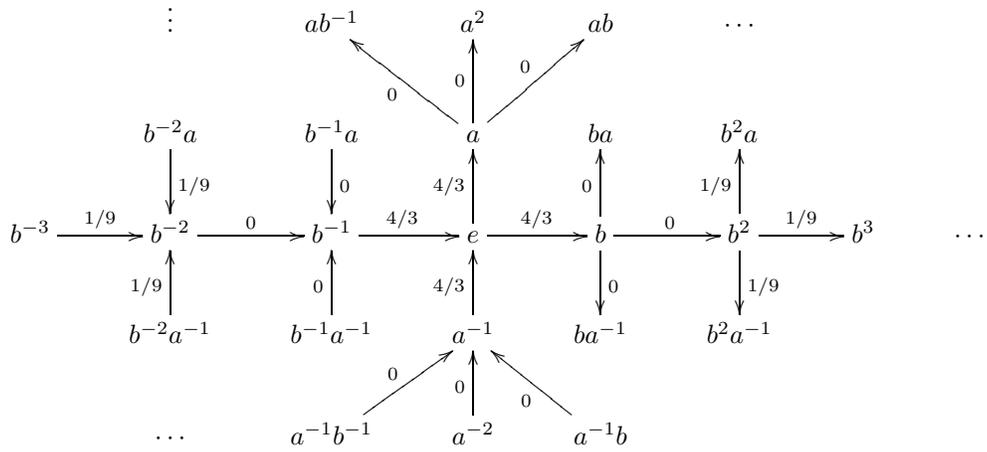
\[\xymatrix{&\ar @{=>}[r]^{\frac{1}{9}(\delta(\chi_{\{b^{2}\}}-\delta(\chi_{\{b^{-2}\}})+\delta(\chi_{\{a^{2}\}})+\cdots) }&}\]
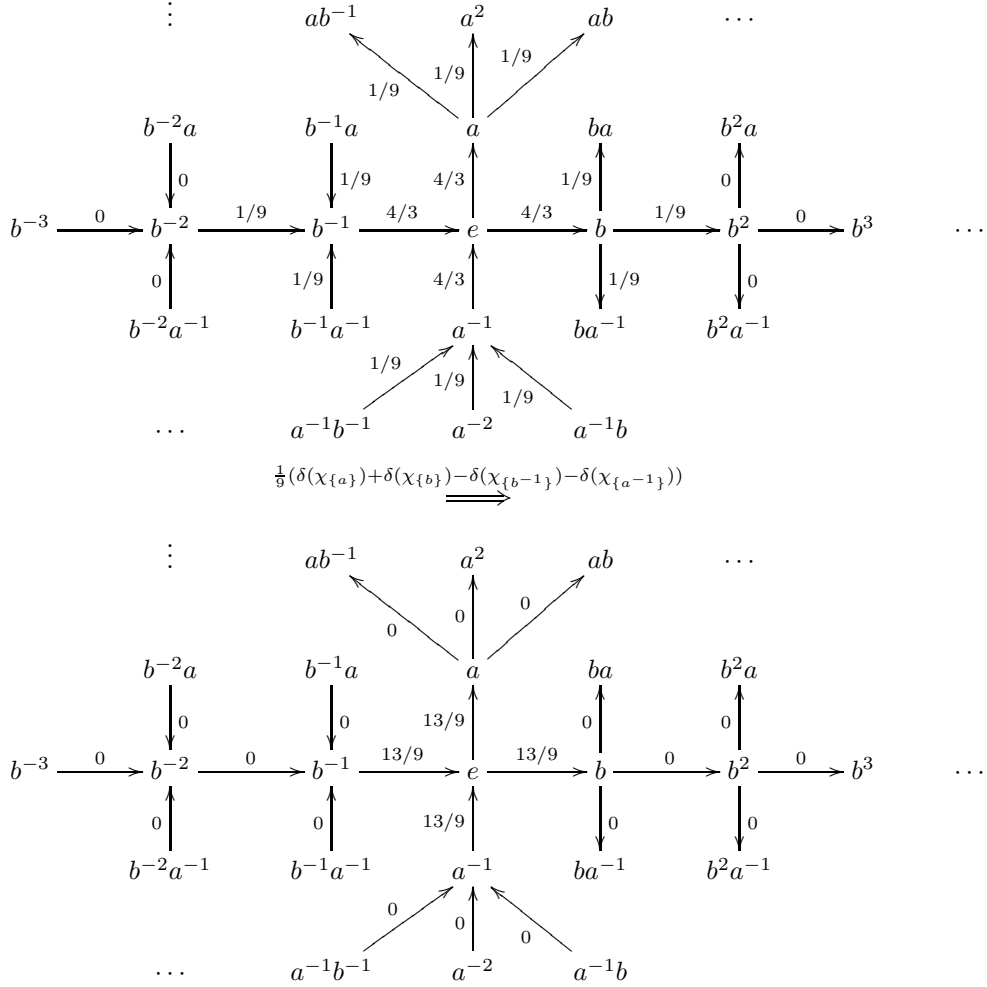
\[\xymatrix{
              &  \vdots  &ab^{-1}                       &             a^{2}                             &ab                         & \cdots  &  & &\\
               &b^{-2}a\ar[d]^{0}   &b^{-1}a \ar[d]^{1/9}          & \ar[ul]^{1/9} a \ar[u]^{1/9} \ar[ur]^{1/9}    &ba                         &b^{2}a   & & &\\
b^{-3}\ar[r]^{0}&b^{-2}\ar[r]^{1/9}&b^{-1}\ar[r]^{4/3}              &\ar[u]^{4/3}e\ar[r]^{4/3}                           &b \ar[d]^{1/9}\ar[u]^{1/9}\ar[r]^{1/9}& b^{2}\ar[u]^{0}\ar[r]^{0}\ar[d]^{0}&b^{3}& \cdots\\
& b^{-2}a^{-1}\ar[u]^{0} &b^{-1}a^{-1}\ar[u]^{1/9}     &\ar[u]^{4/3}a^{-1}                                &ba^{-1}        &b^{2}a^{-1}& & &\\
 &       \cdots            & a^{-1}b^{-1}\ar[ur]^{1/9}  & \ar[u]^{1/9} a^{-2}                            &\ar[ul]^{1/9} a^{-1}b      &&&& }\]
\[\xymatrix{&\ar @{=>}[r]^{\frac{1}{9}(\delta(\chi_{\{a\}})+\delta(\chi_{\{b\}})-\delta(\chi_{\{b^{-1}\}})-\delta(\chi_{\{a^{-1}\}}))} &}\]
\[\xymatrix{
              &  \vdots  &ab^{-1}                       &             a^{2}                             &ab                         & \cdots  &  & &\\
               &b^{-2}a\ar[d]^{0}   &b^{-1}a \ar[d]^{0}          & \ar[ul]^{0} a \ar[u]^{0} \ar[ur]^{0}    &ba                         &b^{2}a   & & &\\
b^{-3}\ar[r]^{0}&b^{-2}\ar[r]^{0}&b^{-1}\ar[r]^{13/9}              &\ar[u]^{13/9}e\ar[r]^{13/9}                           &b \ar[d]^{0}\ar[u]^{0}\ar[r]^{0}& b^{2}\ar[u]^{0}\ar[r]^{0}\ar[d]^{0}&b^{3}& \cdots\\
& b^{-2}a^{-1}\ar[u]^{0} &b^{-1}a^{-1}\ar[u]^{0}     &\ar[u]^{13/9}a^{-1}                                &ba^{-1}       &b^{2}a^{-1}& & &\\
 &       \cdots            & a^{-1}b^{-1}\ar[ur]^{0}  & \ar[u]^{0} a^{-2}                            &\ar[ul]^{0} a^{-1}b      &&&& }\]
Since $\sup_{n}\|g_{n}\|_{p}<\infty$ we find that $g_{n}$ converges weakly to
\[\frac{3}{2}(\mathcal{E}_{(e,a)}+\mathcal{E}_{(e,b)}+\mathcal{E}_{(b^{-1},e)}+\mathcal{E}_{(a^{-1},e)}).\]
Rescaling we find that
\[\mathcal{E}_{(e,a)}+\mathcal{E}_{(e,b)}+\mathcal{E}_{(b^{-1},e)}+\mathcal{E}_{(a^{-1},e)}\in V.\]
By adding $\pm \delta(\chi_{\{e\}})$ and scaling we find that
\begin{align}\label{E:wearealmostdone}
\mathcal{E}_{(e,a)}+\mathcal{E}_{(e,b)}\in V,\\ \nonumber
\mathcal{E}_{(e,b^{-1})}+\mathcal{E}_{(e,a^{-1})}\in V.\nonumber
\end{align}
Since
\begin{align*}
\mathcal{E}_{(e,a)}+\mathcal{E}_{((ba^{-1})^{n},(ba^{-1})^{n}b)}=\mathcal{E}_{(e,a)}+\mathcal{E}_{((ba^{-1})^{n-1},(ba^{-1})^{n-1}b)}&+(ba^{-1})^{n-1}b\left(\mathcal{E}_{(e,b^{-1})}+\mathcal{E}_{(e,a^{-1})}\right)\\
&+(ba^{-1})^{n}\left(\mathcal{E}_{(e,a)}+\mathcal{E}_{(e,b)}\right),
\end{align*}
by $\FF_{2}$-invariance of $V$ we see inductively from (\ref{E:wearealmostdone}) that
\[\mathcal{E}_{(e,a)}+\mathcal{E}_{((ba^{-1})^{n-1},(ba^{-1})^{n-1}b)}\in V\]
for all $n\in\NN.$ If we let $n\to\infty$ we see that
\[\mathcal{E}_{(e,a)}\in \overline{V}^{\textnormal{weak}}=V.\]
Subtracting $\mathcal{E}_{(e,a)}$ from $\mathcal{E}_{(e,a)}+\mathcal{E}_{(e,b)}$ and using (\ref{E:wearealmostdone}) we see that
\[\mathcal{E}_{(e,a)},\mathcal{E}_{(e,b)}\in V.\]
By $\FF_{2}$-invariance of $V$ we see that $V=l^{p}(E(\FF_{2})).$ This completes the proof.

\end{proof}

\begin{theorem} Fix $n\in \NN,$ and a sofic approximation $\Sigma.$

(a) The dimension of the  $l^{p}$-cohomology groups of $\FF_{n}$ satisfy
\[\dim_{\Sigma,l^{p}}(H^{1}_{l^{p}}(\FF_{n}),\FF_{n})=\underline{\dim}_{\Sigma,l^{p}}(H^{1}_{l^{p}}(\FF_{n}),\FF_{n})=n-1,\mbox{ for $1\leq p\leq 2$},\]
\[H^{m}_{l^{p}}(\FF_{n})=\{0\}\mbox{ for $m\geq 2$}.\]

(b) The dimension of the  $l^{p}$-homology groups of $\FF_{n}$ satisfy:
\[\dim_{\Sigma,l^{p}}(H^{l^{p}}_{n}(\FF_{n}),\FF_{n})=\underline{\dim}_{\Sigma,l^{p}}(H^{l^{p}}_{n}(\FF_{n}),\FF_{n})= n-1,\mbox{ for $1<p<2$,}\]
\[H^{l^{1}}_{1}(\FF_{n})=\ker(\del)\cap l^{1}(E(\FF_{n}))=\{0\},\]
\[H^{l^{p}}_{m}(\FF_{n})=0\mbox{ for $m\geq 2.$}\]
\end{theorem}

\begin{proof}
	The statements about higher-dimensional homology or cohomology are clear, since we know that the Cayley graph of $\FF_{n}$ is contractible and one-dimensional.

	Since the image of $\delta$ is closed, the sequence
\[\begin{CD} 0 @>>> l^{p}(\FF_{n})@>\delta>> l^{p}(E(\FF_{n}))  @>>> H^{1}_{l^{p}}(\FF_{n}) @>>> 0\end{CD}\]
is exact. Subadditivity under exact sequences and the computation for $l^{p}$-spaces implies that
\begin{align*}
n
&=\underline{\dim}_{\Sigma,l^{p}}(l^{p}(E(\FF_{n})),\FF_{n})\\
&\leq \underline{\dim}_{\Sigma,l^{p}}(H^{1}_{l^{p}}(\FF_{n}),\FF_{n})+\dim_{\Sigma,l^{p}}(l^{p}(\FF_{n}),\FF_{n})\\
&=\underline{\dim}_{\Sigma,l^{p}}(H^{1}_{l^{p}}(\FF_{n}),\FF_{n})+1.
\end{align*}
Thus
\[\underline{\dim}_{\Sigma,l^{p}}(H^{1}_{l^{p}}(\FF_{n}),\FF_{n})\geq n-1.\]
On the other hand, by  Lemma \ref{L:generateco},  $H^{1}_{l^{p}}(\FF_{n})$ can be generated by $n-1$ elements, so
\[\dim_{\Sigma,l^{p}}(H^{1}_{l^{p}}(\FF_{n}),\FF_{n})\leq n-1,\]
which proves the first claim.

	For the second claim, let $1<p'<\infty$ be such that $\frac{1}{p}+\frac{1}{p'}=1.$  Note that Lemma \ref{L:RIHTDT} implies that $\delta\colon l^{p'}(\FF_{n})\to l^{p'}(E(\FF_{n}))$ is an injection with closed image. Taking transposes, we see that $\delta\colon l^{p}(E(\FF_{n}))\to l^{p}(\FF_{n})$ is surjective. Thus the sequence
\[\begin{CD} 0 @>>> H^{l^{p}}_{1}(\FF_{n})@>>> l^{p}(E(\FF_{n}))  @>\del>> l^{p}(\FF_{n}) @>>> 0\end{CD}\]
is exact. As in the first half this implies that
\[\underline{\dim}_{\Sigma,l^{p}}(H^{l^{p}}_{1}(\FF_{n}),\FF_{n})\geq n-1\]
for $1<p\leq 2.$ The upper bound for $1<p\leq 2$ also holds by the preceding proposition.

	We turn to the last claim. Because the Cayley graph of $\FF_{n}$ is a tree, for $x\in \FF_{n}$ we can define $\gamma_{x}$ to be the unique geodesic path from $e$ to $x.$ Define
\[A\colon \CC^{E(\FF_{n})}\to \CC^{\FF_{n}}\]
 by
\[(Af)(x)=\sum_{j=1}^{|x|}f(\gamma_{x}(j-1),\gamma_{x}(j)).\]
Note that $\delta(Af)=f.$ A direct computation verifies that $A(\E_{(x,xa_{j})})\in l^{\infty}(\FF_{n}),$ so $\delta(l^{\infty}(\FF_{n}))$ is weak$^{*}$ dense in $l^{\infty}(E(\FF_{n})).$ By duality $\ker(\partial)\cap l^{1}(E(\FF_{n}))=\{0\}.$ This completes the proof.

\end{proof}

\section{Closing Remarks}

	Here are some natural conjectures based on our work in this paper and \cite{Me}.

\begin{conj}\label{C:Gor} Let $\Gamma$ be a an amenable group and $W\subseteq l^{p}(\Gamma)^{\oplus n}$ for some $n\in \NN.$ Let $\dim_{l^{p}}^{G}(W,\Gamma)$ be $l^{p}$-dimension as defined by Gournay in \cite{Gor}. Then for any sofic approximation $\Sigma$ of $\Gamma$ we have
\[\dim_{l^{p}}^{G}(W,\Gamma)=\dim_{\Sigma,l^{p}}(W,\Gamma).\]\end{conj}

\begin{conj}\label{C:dim} Let $2<p<\infty,$ and let $\Gamma$ be a countable discrete sofic group with sofic approximation $\Sigma.$ Then for all $n\in \NN,$
\[\dim_{\Sigma,l^{p}}(l^{p}(\Gamma)^{\oplus n},\Gamma)=\underline{\dim}_{\Sigma,l^{p}}(l^{p}(\Gamma)^{\oplus n},\Gamma)=n.\]
\end{conj}

 Little progress has been made on Conjectures \ref{C:Gor},\ref{C:dim}. It may be quite possible that our definition is simply not the right way to look at von Neumann dimension for the action of $\Gamma$ on $l^{p}(\Gamma)^{\oplus n}$ if $2< p<\infty$. Another natural conjecture based on the techniques in Section \ref{S:trivial} is the following.

\begin{conj} Let $\Gamma$ be an amenable group. If $\Sigma,\Sigma'$ are two sofic approximations of $\Gamma$ and $V$ is a uniformly bounded representation of $\Gamma,$ then
\[\dim_{\Sigma,l^{p}}(V,\Gamma)=\dim_{\Sigma',l^{p}}(V,\Gamma).\]
\end{conj}
Because of the techniques in  Section $\ref{S:trivial},$ if $\Sigma=(\sigma_{i}\colon\Gamma\to S_{d_{i}}),$ it suffices to assume $\Sigma=(\sigma^{\oplus k_{i}}),$ for a sequence of integers $k_{i}.$

$\mathbf{Acknowledgment}.$ The author would like to thank Dimitri Shlyakhtenko for his helpful advice on the problem. The author would like to thank the anonymous referee whose numerous comments greatly improved the paper.

\end{document}